\documentclass[reqno,a4paper,twoside,draft]{amsart}
\usepackage[a4paper,margin=1in,footskip=0.25in]{geometry}

\usepackage{enumitem}

\setenumerate{label=\textnormal{(\arabic*)}}

\usepackage{amsmath,amssymb,dsfont,verbatim,mathtools,bm,geometry,fge,endnotes}

\usepackage[raggedright]{titlesec}
\usepackage{booktabs}

\usepackage{multirow}
\usepackage[latin1]{inputenc}
\usepackage{tikz}
\usepackage{float,subfig}
\usepackage{amsrefs}

\titleformat{\section}
{\normalfont\Large\bfseries\center}{\thesection}{1em}{}
\titleformat{\subsection}
{\normalfont\large\bfseries}{\thesubsection}{1em}{}
\titleformat{\subsubsection}
{\normalfont\normalsize\bfseries}{\thesubsubsection}{1em}{}
\titleformat{\paragraph}[runin]
{\normalfont\normalsize\bfseries}{\theparagraph}{1em}{}
\titleformat{\subparagraph}[runin]
{\normalfont\normalsize\bfseries}{\thesubparagraph}{1em}{}

\titlespacing*{\chapter} {0pt}{50pt}{40pt}
\titlespacing*{\section} {30pt}{3.5ex plus 1ex minus .2ex}{2.3ex plus .2ex}
\titlespacing*{\subsection} {0pt}{3.25ex plus 1ex minus .2ex}{1.5ex plus .2ex}
\titlespacing*{\subsubsection}{0pt}{3.25ex plus 1ex minus .2ex}{1.5ex plus .2ex}
\titlespacing*{\paragraph} {0pt}{3.25ex plus 1ex minus .2ex}{1em}
\titlespacing*{\subparagraph} {\parindent}{3.25ex plus 1ex minus .2ex}{1em}
\usepackage{titletoc}
\contentsmargin{2.55em}
\dottedcontents{section}[1.5em]{}{2.8em}{1pc}
\dottedcontents{subsection}[3.7em]{}{3.4em}{1pc}
\dottedcontents{subsubsection}[7.6em]{}{3.2em}{1pc}
\dottedcontents{paragraph}[10.3em]{}{3.2em}{1pc}

\newcommand{\hs}{\hspace{-0.5pt}}

\def\xcirc{\hs\circ\hs}

\newtheorem{theorem}{Theorem}[section]
\newtheorem{lemma}[theorem]{Lemma}
\newtheorem{proposition}[theorem]{Proposition}
\newtheorem{corollary}[theorem]{Corollary}

\theoremstyle{definition}
\newtheorem{definition}[theorem]{Definition}

\newtheorem{example}[theorem]{Example}

\theoremstyle{remark}
\newtheorem{remark}[theorem]{Remark}

\DeclareMathOperator{\scc}{SCC}
\newcommand{\Coalg}{\textbf{Coalg}}

\DeclareMathOperator{\ord}{ord}

\DeclareMathOperator{\ide}{id}

\DeclareMathOperator{\End}{End}
\newcommand{\ov}{\overline}
\newcommand{\ot}{\otimes}

\newcommand{\dpu}{\mathbin{:}}

\begin{document}
\title[Cocommutative $q$-cycle coalgebra structures]{Cocommutative $q$-cycle coalgebra structures on the dual of the truncated polynomial algebra}

\author{Jorge A. Guccione}
\address{Departamento de Matem\'atica\\ Facultad de Ciencias Exactas y Naturales-UBA, Pabell\'on~1-Ciudad Universitaria\\ Intendente Guiraldes 2160 (C1428EGA) Buenos Aires, Argentina.}
\address{Instituto de Investigaciones Matem\'aticas ``Luis A. Santal\'o"\\ Facultad de Ciencias Exactas y Natu\-ra\-les-UBA, Pabell\'on~1-Ciudad Universitaria\\ Intendente Guiraldes 2160 (C1428EGA) Buenos Aires, Argentina.}
\email{vander@dm.uba.ar}

\author{Juan J. Guccione}
\address{Departamento de Matem\'atica\\ Facultad de Ciencias Exactas y Naturales-UBA\\ Pabell\'on~1-Ciudad Universitaria\\ Intendente Guiraldes 2160 (C1428EGA) Buenos Aires, Argentina.}
\address{Instituto Argentino de Matem\'atica-CONICET\\ Savedra 15 3er piso\\ (C1083ACA) Buenos Aires, Argentina.}
\email{jjgucci@dm.uba.ar}

\thanks{Jorge A. Guccione and Juan J. Guccione were supported by UBACyT 20020150100153BA (UBA).}

\author[C. Valqui]{Christian Valqui}
\address{Pontificia Universidad Cat\'olica del Per\'u, Secci\'on Matem\'aticas, PUCP, Av. Universitaria 1801, San Miguel, Lima 32, Per\'u.}
\address{Instituto de Matem\'atica y Ciencias Afines (IMCA) Calle Los Bi\'ologos 245. Urb San C\'esar.
La Molina, Lima 12, Per\'u.}
\email{cvalqui@pucp.edu.pe}

\begin{abstract}
  In order to construct solutions of the braid equation we consider bijective left non-degenerate set-theoretic type solutions,
which correspond to regular q-cycle coalgebras.  We obtain a partial classification of the different q-cycle coalgebra structures on the dual coalgebra of $K[y]/\langle y^n\rangle$, the truncated polynomial algebra. We obtain an interesting family of involutive q-cycle coalgebras which we call Standard Cycle Coalgebras.  They are parameterized by free parameters
$\{p_1,...,p_{n-1}\}$ and in order to verify that they are compatible with the braid equation, we have to
verify that certain differential operators $\partial^j$ on formal power series in two variables  $K[[x, y]]$ satisfy the condition
 $(\partial^j G)_i = (\partial^i G)_j$ for all i, j, where $G$ is a formal power series associated to the given q-cycle coalgebra. It would be interesting
to find out the relation of these operators with the operators given by Yang in the context with 2-dimensional quantum field theories, which was one of the
origins of the Yang-Baxter equation.
\end{abstract}

\maketitle

\tableofcontents

\section*{Introduction}
Since the Yang-Baxter equation first appeared in 1967 in the paper \cite{Y} by Yang, and then in 1972 in the paper \cite{B} by
Baxter, many variations of the Yang-Baxter equation have been constructed by physicists and mathematicians. In the present article we will consider
the approach to the Yang-Baxter equation from the point of view of the braid equation.

Let $V$ be a vector space over a field $k$ and let $s\colon V \ot V \to V \ot V$ be a linear map. We say that $s$ satisfies the braid equation if
\begin{equation}\label{ecuacion de trenzas}
s_{12} \xcirc s_{23} \xcirc s_{12} = s_{23} \xcirc s_{12} \xcirc s_{23},
\end{equation}
where $s_{ij}$ means $s$ acting in the $i$-th and $j$-th components.
In~\cite{D}, Drinfeld raised the question of finding set-theoretical (or combinatorial) solutions; i.e. pairs $(Y,s)$, where $Y$ is a set and
$s\colon Y\times Y\to Y\times Y$ is a map satisfying~\eqref{ecuacion de trenzas}.
	
This approach was first considered by Etingof, Schedler and Soloviev~\cite{ESS} and Gateva-Ivanova and Van den Bergh~\cite{GIVB}, for involutive
solutions, and by Lu, Yan and Zhu~\cite{LYZ}, and Soloviev~\cite{S}, for non-involutive solutions. Now it is known that there are connections between
solutions and affine torsors, Artin-Schelter regular rings, Biberbach groups and groups of $I$-type, Garside structures, Hopf-Galois theory, left
symmetric algebras, etc. (\cites{AGV, DG, De1, De2, DDM, Ch, GI1, GI2, GI3, GI4, GIVB, JO}).

In~\cite{GGVa}, the relationships between set-theoretical solutions, $q$-cycle sets, $q$-braces, skew-braces, matched pairs of groups and invertible
$1$-cocycles considered in the set-theoretical setting in \cites{CJO, ESS, GV, LYZ, R1}, were generalized to the setting of non-cocommutative
coalgebras following the ideas of Rump in~\cite{R1} and~\cite{R2}.

We take the definition of $q$-cycle coalgebra of~\cite{GGVa}, and
analyze the different $q$-cycle coalgebra structures on cocommmutative coalgebras, in particular on $C$, the dual coalgebra of $K[y]/(y^n)$,
the truncated polynomial algebra.
We achieve a partial classification of such structures, in particular we find a family of involutive $q$-cycle coalgebras which we call {\bf S}tandard
{\bf C}ycle {\bf C}oalgebras ($\scc$, see Definition~\ref{SIQ def}).

Denote by $\cdot$ and $:$ the $q$-cycle coalgebra operations on $(C,\frak{p},\frak{d})$, and write
$$
x_i\cdot x_j=\sum_k \frak{p}_{ij}^k x_k\quad\text{and}\quad  x_i : x_j=\sum_k \frak{d}_{ij}^k x_k,
$$
where $\{x_0,x_1,x_2,\dots,x_{n-1}\}$ is the dual basis of  $\{1,y,y^2,\dots,y^{n-1}\}$.
The associated solution of the braid equation is involutive if and only if $\frak{p}=\frak{d}$.
An involutive $q$-cycle algebra is called a cycle coalgebra and in the case of a $\scc$ all the coefficients $\frak{p}_{ij}^k$ depend only on the
values $\frak{p}_{i1}^1$, for $i=1,\dots,n-1$, and the degree of a
$\scc$ is defined as $i_0=\min\{i>0, \frak{p}_{i1}^1\ne 0\}$.
Moreover, in this case the element $x_0$, corresponding to the element $y^0=1$ in the truncated polynomial
algebra, acts as the identity via $\cdot$, i.e., we have $x\cdot x_0=x$ for all $x$. Note that $x_0$ acts as identity via $\cdot$
if and only if $\frak{p}_{j0}^k=\delta_{kj}$ for all $j,k$. One can show, using that $\frak{p}$ is compatible with the comultiplication, that
this is equivalent to $\frak{p}_{j0}^1=\delta_{1j}$ for all $j$.

One of our main results is the construction of a unique $\scc$ of degree $i_0$, starting from the values of the free parameters
$\{\frak{p}_{1i}^1\}_{i=i_0,\dots,n-1}$. For this we construct several differential operators on the ring of formal power series in two variables,
among them certain operators called $\partial^j$, and show that the braid equation is satisfied for $\frak{p}$ if and only if
$(\partial^j G)_i=(\partial^i G)_j$ for all $i,j$, where the subindex denotes the homogeneous component with respect to the
$y$-degree of a formal power series in $K[[x,y]]$, and $G(x,y)=\sum_{i,j} \frak{p}_{ij}^1 x^i y^j$ (See Remark~\ref{remark importante1}).

Another rather surprising result is the fact that if $\frak{p}_{11}^1\ne 0$, then we automatically have a $\scc$ (of degree $1$).
In general, the structure of the $q$-cycle coalgebra depends on the value of $\frak{p}_{10}^1$ which in the case of a $\scc$ is equal to $1$.
On the other end of the classification table we obtain a complete
classification result, if we consider the case when $\frak{p}_{10}^1$ is not a root of the unity of order less than $n$, or
the case when $\frak{d}_{10}^1$ is not a root of the unity of order less than $n$.

\begin{table}[htb] CLASSIFICATION TABLE\vspace{0.5cm}\\
  \begin{tabular}{| p{1.2cm}| p{0.9cm} | l |  l |  l | l | l |}
    \hline
    \multicolumn{5}{|c|}{Conditions} & Results  & Reference\\ \hline
    $\frak{p}_{11}^1\ne 0$ &\multicolumn{4}{|c|}{In this case automatically $\frak{p}=\frak{d}$ and $\frak{p}_{j0}^1=\delta_{1j}$}&CC,$\scc$ &
    Cor. \ref{cor i0 igual a 1}  \\   \hline
    \multirow{8}{1.2cm}{$\frak{p}_{11}^1 = 0$}
    & \multirow{6}{0.9cm}{$\frak{p}=\frak{d}$} & \multirow{2}{1.4cm}{$\frak{p}_{j0}^1=\delta_{1j}$}&
    \multicolumn{2}{|c|}{$\exists j>1, \frak{p}_{1j}^1\ne 0$}
    &  CC,$\scc$ & Cor. \ref{cor i0 mayor que 1} \\ \cline{4-7}
    && &\multicolumn{2}{|c|}{$\forall j>0, \frak{p}_{1j}^1 = 0$}& PR + Ex  & P.~\ref{caso casi SIQ}, Ex.~\ref{ej casi SIQ involutivo}\\ \cline{3-7}
    &   & \multirow{3}{1.4cm}{$\frak{p}_{j0}^1\ne \delta_{1j}$} &$\frak{p}_{10}^1=1$&& Ex & Ex.~\ref{ej casi SIQ involutivo} \\ \cline{4-7}
    &   & & \multirow{2}{1.15cm}{$\frak{p}_{10}^1\ne 1$}&$\exists k<n, (\frak{p}_{10}^1)^k=1$&PR + Ex &T.~\ref{teorema no raiz de la unidad},
    Ex.~\ref{ejemplo involutivo} \\ \cline{5-7}
    &   & &&$\forall k<n, (\frak{p}_{10}^1)^k\ne 1$&CC& Cor. ~\ref{ejemplo completo no raiz de la unidad}\\ \cline{2-7}
    & \multirow{3}{0.9cm}{$\frak{p}\ne \frak{d}$}   & \multicolumn{3}{|c|}{$\exists k_1,k_2<n,\ (\frak{p}_{10}^1)^{k_1}=1,
    (\frak{d}_{10}^1)^{k_2}=1$}&PR + Ex&T.~\ref{teorema no raiz de la unidad}, R.~\ref{simetrico}, Ex.~\ref{ejemplo no involutivo}\\ \cline{3-7}
    &   & \multicolumn{3}{|c|}{$\forall k<n,\ (\frak{p}_{10}^1)^{k}\ne 1$}&CC&Cor. ~\ref{ejemplo completo no raiz de la unidad}\\ \cline{3-7}
    &   & \multicolumn{3}{|c|}{$\forall k<n,\ (\frak{d}_{10}^1)^{k}\ne 1$}&CC&Cor. ~\ref{ejemplo completo no raiz de la unidad},
    R.~\ref{simetrico}\\ \hline
  \end{tabular}\vspace{0.3cm}
  CC=Complete Classification,\quad PR=Partial Results,\quad Ex=Examples
\end{table}

\newpage
\setcounter{equation}{0}
\section{Preliminaries}
Let $C$ be a cocommutative coalgebra and let $s\in\End_{\Coalg}(C^2)$, where $C^2$ denotes the coalgebra $C\ot C$. The maps
$s_1\coloneqq (C\ot\epsilon)\circ s$ and $s_2\coloneqq (\epsilon\ot C) \circ s$ are the unique coalgebra morphisms such that
$s=(s_1\ot s_2)\circ\Delta_{C^2}$. We will write ${}^ab\coloneqq s_1(a\ot b)$ and $a^b\coloneqq s_2(a\ot b)$.

\begin{definition}\label{def de left no degenerado}
  A coalgebra endomorphism $s$ of $C^2$, is called {\em left non-degenerate} if there exists a map $\mathfrak{p}\colon C^2\to C$ such that
  \begin{equation}\label{no deg a izq}
    a^{b_{(1)}}\cdot b_{(2)}=(a\cdot b_{(1)})^{b_{(2)}}=\epsilon(b)a,
  \end{equation}
  where $a\cdot b\coloneqq\mathfrak{p}(a\ot b)$. In the sequel we write $\mathfrak{d}(a\ot b)=a\dpu b\coloneqq {}^{b\cdot a_{(1)}}a_{(2)}$. Note that
  the map $\mathfrak{p}$ is unique and that $\mathfrak{p}$ and $\mathfrak{d}$ are coalgebra morphisms from $C^2$ to $C$.
\end{definition}

\begin{definition}\label{YBE}
  A $K$-linear map $s\colon C^2\to C^2$ is a {\em solution of the braid equation} if
  $$
    s_{12}\circ s_{23}\circ s_{12}=s_{23}\circ s_{12}\circ s_{23},
  $$
  where $s_{12}\coloneqq s\ot C$ and $s_{23}\coloneqq C\ot s$. If $s$ is also a coalgebra endomorphism of $C^2$, then we say that $s$ is
  a {\em set-theoretic type} solution.
\end{definition}

\begin{definition}\label{q cycle coalgebra}
  A triple $\mathcal{C}\coloneqq (C,\mathfrak{p},\mathfrak{d})$, where $\mathfrak{p}$ and $\mathfrak{d}$ are coalgebra morphisms from $C^2$ to $C$,
  is a {\em regular $q$-magma coalgebra over $C$} if there exist coalgebra morphisms $a\ot b\mapsto a^b$ and $a\ot b\mapsto a_b$, from $C^2$ to $C$,
  such that identities~\eqref{no deg a izq} are fulfilled, and
  \begin{equation}\label{reg a derecha}
    (a\dpu b_{(2)})_{b_{(1)}}=a_{b_{(2)}}\dpu b_{(1)}=\epsilon(b)a\quad\text{for all $a,b\in C$,}
  \end{equation}
  where $a\cdot b\coloneqq \mathfrak{p}(a\ot b)$ and $a\dpu b\coloneqq \mathfrak{d}(a\ot b)$. We say that $\mathcal{C}$ is
  a {\em $q$-cycle coalgebra} if and the following conditions hold for all $a,b,c\in C$:
  \begin{enumerate}[itemsep=0.7ex, topsep=1.0ex, label={\arabic*.}]

    \item $(a\cdot b_{(1)})\cdot (c\dpu b_{(2)})=(a\cdot c_{(2)})\cdot (b\cdot c_{(1)})$,

    \item $(a\cdot b_{(1)})\dpu (c\cdot b_{(2)})=(a\dpu c_{(2)})\cdot (b\dpu c_{(1)})$,

    \item $(a\dpu  b_{(1)})\dpu (c\dpu b_{(2)})=(a\dpu c_{(2)})\dpu (b\cdot c_{(1)})$.

  \end{enumerate}

\end{definition}

\begin{remark}
  By~\cite{GGV}*{Remarks~1.12 and 1.15} the maps $a\ot b\mapsto a^b$ and $a\ot b\mapsto a_b$ are unique, and the maps
  $G_p(a\ot b)\coloneqq a\cdot b_{(1)}\ot b_{(2)}$ and $G_d(a\ot b)\coloneqq a\dpu b_{(2)}\ot b_{(1)}$ are bijective.
\end{remark}

\noindent The following result is a particular case of~\cite{GGVa}*{Propositions 2.4, 2.6 and Remark 2.7}

\begin{theorem}\label{soluciones de tipo conjuntista=q-cycle coalgebras}
  If $s\in\End_K(C^2)$ is a bijective left non-degenerate set-theoretic type solution of the braid equation, then
  $\mathcal{C}\coloneqq (C,\mathfrak{p},\mathfrak{d})$, where $\mathfrak{p}$ and $\mathfrak{d}$ are the maps introduced in
  Definition~\ref{def de left no degenerado}, is a regular $q$-cycle coalgebra. Conversely, if $\mathcal{C}\coloneqq (C,\mathfrak{p},\mathfrak{d})$
  is a regular $q$-cycle coalgebra, then the map $s\colon C^2\to C^2$, given by $s(a\ot b)\coloneqq {}^{a_{(1)}}b_{(2)}\ot {a_{(2)}}^{b_{(1)}}$,
  where $a^b$ is as in Definition~\ref{q cycle coalgebra} and ${}^ab\coloneqq b_{(2)}\dpu a^{b_{(1)}}$, is a bijective left non-degenerate
  set-theoretic type solution of the braid equation.
\end{theorem}

\begin{definition}\label{cycle coalgebra}
  A pair $\mathcal{C}\coloneqq (C,\mathfrak{p})$ is a {\em regular magma coalgebra over $C$} if $(C,\mathfrak{p},\mathfrak{p})$ is a regular
  $q$-magma coalgebra. We say that $\mathcal{C}$ is a {\em cycle coalgebra} if $(C,\mathfrak{p},\mathfrak{p})$ is a $q$-cycle coalgebra. In
  this case conditions 1, 2 and~3 of Definition~\ref{q cycle coalgebra} reduce to the single condition
  $(a\cdot b_{(1)})\cdot (c\cdot b_{(2)})=(a\cdot c_{(2)})\cdot (b\cdot c_{(1)})$.
\end{definition}

\begin{remark}\label{soluciones involutivas=cycle coalgebras}
  A $q$-cycle coalgebra is a cycle coalgebra if and only if its associated solution
  according to Theorem~\ref{soluciones de tipo conjuntista=q-cycle coalgebras} is involutive.
\end{remark}

\subsection{Morphisms of the dual of the truncated polynomial algebra}
From now on $C$ is the dual  coalgebra of the algebra $K[y]/\langle y^n\rangle$, where $K$ is an algebraically closed field of characteristic~$0$.
So, if $\{x_0,x_1,x_2,\dots,x_{n-1}\}$ is the dual basis of  $\{1,y,y^2,\dots,y^{n-1}\}$, we have
$$
  \Delta(x_i)=\sum_{j+k=i} x_j\otimes x_k\quad\text{and}\quad\epsilon(x_i)=\delta_{i0}.
$$
Let $\mathfrak{p},\mathfrak{d}\colon C^2\to C$ be two maps. Write $\mathfrak{p}(x_i\ot x_j)=\sum_k\mathfrak{p}_{ij}^k x_k$ and
$\mathfrak{d}(x_i\ot x_j)=\sum_k\mathfrak{d}_{ij}^k x_k$. For convenience we set $\mathfrak{d}_{ij}^k=\mathfrak{p}_{ij}^k=0$ for
$k\ge n$, $i<0$ or $j<0$.

\begin{remark}\label{cond para q-magma}
  A direct computation shows that the maps $\mathfrak{p}$ and $\mathfrak{d}$ are coalgebra morphisms if and only if
  $$
    \mathfrak{p}_{ij}^0 = \mathfrak{d}_{ij}^0 =\delta_{0,i+j},\quad \mathfrak{p}_{ij}^k =
    \sum_{\substack{a+b=i\\ c+d=j}}\mathfrak{p}_{ac}^l\mathfrak{p}_{bd}^h \quad\text{and}\quad \mathfrak{d}_{ij}^k=\sum_{\substack{a+b=i\\c+d=j}}
    \mathfrak{d}_{ac}^l\mathfrak{d}_{bd}^h,
  $$
  for all $i,j,k,l,h\ge 0$ such that $l+h=k$.
\end{remark}

In the rest of this subsection we assume that $\mathfrak{p}$ and $\mathfrak{d}$ are coalgebra maps.

\begin{remark}\label{primeras condiciones}
  Since $x_0\otimes x_0$ is the unique group like element in $C^2$ and $x_0$ is the unique group like element in $C$, we have
  $\mathfrak{p}_{00}^j =\mathfrak{d}_{00}^j =\delta_{0j}$.
\end{remark}

\begin{lemma}\label{filtrado}
  We have $\mathfrak{p}_{ij}^k=\mathfrak{d}_{ij}^k=0$ for $k>i+j$.
\end{lemma}

\begin{proof}
  It suffices to prove it for $\mathfrak{p}$, since the same argument applies to $\mathfrak{d}$. We proceed by induction on $r\coloneqq i+j$. If
  $r=0$, this is true by the previous remark. Assume it is true for $r<r_0$ and take $i,j$ with $i+j=r_0$. On one hand, by inductive hypothesis
  \begin{equation}\label{suma grande}
    (\mathfrak{p}\otimes\mathfrak{p})\circ\Delta(x_i\otimes x_j)=\sum_{k=0}^{n-1}\mathfrak{p}_{ij}^k x_k\otimes
    x_0+x_0\otimes\sum_{k=0}^{n-1}\mathfrak{p}_{ij}^k x_k +\sum_{\substack {a+b=i\\ c+d=j\\ a+c<i+j\\ b+d<
    i+j}}\sum_{l=0}^{a+c}\sum_{h=0}^{b+d}\mathfrak{p}_{ac}^l x_l\otimes\mathfrak{p}_{bd}^h x_h.
  \end{equation}
  On the other hand
  $$
    \Delta\circ\mathfrak{p}(x_i\otimes x_j)=\sum_{k=0}^{n-1}\sum_{l+h=k}\mathfrak{p}_{ij}^k x_l\otimes x_h.
  $$
  Since $\mathfrak{p}$ is a coalgebra map both expressions coincide. But if $k>i+j\ge 1$, then the coefficient of $x_1\ot x_{k-1}$
  in~\eqref{suma grande} equals $0$, because if $x_l\ot x_h$ has non zero coefficient, then $l+h\le a+c+b+d=i+j<k$, and so $\mathfrak{p}_{ij}^k=0$
  for $k>i+j$, concluding the proof.
\end{proof}

\begin{remark}\label{remark grande}
  Let $i,j,l,h<n$ and set $k\coloneqq l+h$. By Remark~\ref{cond para q-magma} and Lemma~\ref{filtrado}, we have
  \begin{equation}\label{eq: comultiplicative map}
    \mathfrak{p}_{ij}^k =\sum_{\substack{a+b=i\\ c+d=j}}\mathfrak{p}_{ac}^l\mathfrak{p}_{bd}^h=\sum_{\substack{a+b=i\\c+d=j\\a+c\ge l\\b+d\ge
    {h}}}\mathfrak{p}_{a c}^l \mathfrak{p}_{bd}^h\quad\text{and}\quad
    \mathfrak{d}_{ij}^k=\sum_{\substack{a+b=i\\c+d=j}}\mathfrak{d}_{ac}^l\mathfrak{d}_{bd}^h=\sum_{\substack{a+b=i\\c+d=j\\a+c\ge l\\
    b+d\ge h}}\mathfrak{d}_{a c}^l \mathfrak{d}_{bd}^h.
  \end{equation}
  From this it follows that
  \begin{equation}\label{producto}
    \mathfrak{p}_{ij}^k=\sum_{\substack{i_1+\dots+i_k=i\\ j_1+\dots+j_k=j\\ i_1+j_1,\dots, i_s+j_s
    \ge 1}}\prod_{s=1}^k\mathfrak{p}_{i_sj_s}^1\quad\text{and} \quad \mathfrak{d}_{ij}^k=\sum_{\substack{i_1+\dots+i_k=i\\ j_1+\dots+j_k=j\\
    i_1+j_1,\dots, i_s+j_s\ge 1}}\prod_{s=1}^k\mathfrak{d}_{i_sj_s}^1.
  \end{equation}
  Consequently,
  \begin{equation}\label{rho_ij^i+j}
    \mathfrak{p}_{ij}^{i+j}=\binom{i+j}{i}(\mathfrak{p}_{10}^1)^i(\mathfrak{p}_{01}^1)^j\quad\text{and}\quad\mathfrak{d}_{ij}^{i+j}=
    \binom{i+j}{i}(\mathfrak{d}_{10}^1)^i (\mathfrak{d}_{01}^1)^j.
  \end{equation}
\end{remark}

\begin{remark}\label{rho_10=0 o rho_01=0}
  By equality~\eqref{rho_ij^i+j}, we have $0=\mathfrak{p}_{n-1,1}^n=n(\mathfrak{p}_{10}^1)^{n-1}\mathfrak{p}_{01}^1$. Thus $\mathfrak{p}_{10}^1=0$ or
  $\mathfrak{p}_{01}^1 =0$. Similarly, $\mathfrak{d}_{10}^1=0$ or $\mathfrak{d}_{01}^1=0$.
\end{remark}

\begin{remark}\label{no degenerado a izquierda}
  The vector space $C^2$ is canonically graded. Moreover, by Lemma~\ref{filtrado}, the maps $G_p$ and $G_d$ are compatible with the associated
  filtration. Their associated graded maps are
  $$
    \tilde G_p(x_i\ot x_j)=\sum_{a+b=j}\mathfrak{p}_{ia}^{i+a} x_{i+a}\ot x_b\quad\text{and}\quad\tilde G_d(x_i\ot
    x_j)=\sum_{a+b=j}\mathfrak{d}_{ia}^{i+a} x_{i+a}\ot x_b.
  $$
  Recall that $(C,\mathfrak{p},\mathfrak{d})$ is a regular $q$-magma coalgebra if and only if $G_p$ and $G_d$ are bijective. By a standard argument
  this happens if and only if $\tilde G_p$ and $\tilde G_d$ are, which occurs if and only if $\mathfrak{p}_{i0}^i\ne 0$ and $\mathfrak{d}_{i0}^i\ne 0$
  for all $i>0$. Hence, by Remark~\ref{rho_10=0 o rho_01=0} we have $\mathfrak{p}_{01}^1=\mathfrak{d}_{01}^1=0$, and so, by
  equalities~\eqref{rho_ij^i+j}
  \begin{equation}\label{exacto se anula}
    \mathfrak{p}_{i0}^i = (\mathfrak{p}_{10}^1)^i,\qquad \mathfrak{d}_{i0}^i = (\mathfrak{d}_{10}^1)^i\qquad\text{and}\qquad
    \mathfrak{p}_{ij}^{i+j}=\mathfrak{d}_{ij}^{i+j}=0\quad\text{if~$j>0$.}
  \end{equation}
\end{remark}

In the rest of this subsection we assume that $(C,\mathfrak{p},\mathfrak{d})$ is a regular $q$-magma coalgebra.

\begin{proposition}\label{lado derecho}
  If $k>i$, then $\mathfrak{p}_{ij}^k=\mathfrak{d}_{ij}^k=0$.
\end{proposition}

\begin{proof}
  We will prove the result only for $\mathfrak{p}_{ij}^k$, since the result for $\mathfrak{d}_{ij}^k$ follows from the same argument. We proceed by
  induction on $j$. For $j=0$ this is true, since $\mathfrak{p}_{i0}^k=0$ by Lemma~\ref{filtrado}. Assume that $j=1$. If  $k>i+1$, then, again by
  Lemma~\ref{filtrado}, we have $\mathfrak{p}_{i1}^k=0$; while, if $k=i+1$, then~\eqref{exacto se anula} gives $\mathfrak{p}_{i1}^k=0$. Assume that
  $j>1$ and $\mathfrak{p}_{il}^k=0$ for $l<j$ and $k>i$. Set $h\coloneqq k-i$. When $h=j$, then by~\eqref{exacto se anula} we know that
  $\mathfrak{p}_{ij}^k=0$. Fix $0<h<j$ and assume that $\mathfrak{p}_{ij}^{i+h'}=0$  when $h'>h$. By~\eqref{eq: comultiplicative map}, we have
  \begin{equation*}\label{suma de productos}
    \mathfrak{p}_{ij}^{i+h}=\sum_{\substack{a+b=i\\ c+d=j\\ a+c\ge 1\\ b+d\ge i+h-1}}\mathfrak{p}_{a c}^1\mathfrak{p}_{b d}^{i+h-1}
    =\sum_{\substack{0\le a\le i\\ 0\le d\le j\\ a+h-1\le d < a+j}}\mathfrak{p}_{a,j-d}^1\mathfrak{p}_{i-a,d}^{i+h-1},
  \end{equation*}
  where the last equality holds because $a+c\ge 1\Leftrightarrow d < a+j$ and $b+d\ge i+h-1\Leftrightarrow d\ge a+h-1$. Thus,
  $$
    \mathfrak{p}_{ij}^{i+h}=\sum_{h-1\le d < j}\mathfrak{p}_{0,j-d}^1\mathfrak{p}_{i,d}^{i+h-1} +\sum_{\substack{0< a\le i\\ a+h-1\le d\le j}}
    \mathfrak{p}_{a,j-d}^1\mathfrak{p}_{i-a,d}^{i+h-1}.
  $$
  Using now the inductive hypothesis and that $\mathfrak{p}_{ij}^{i+h'}=0$ when $h'>h$, we obtain that
  \begin{equation}\label{inductivo para d mayor que uno}
    \mathfrak{p}_{ij}^{i+h}=\begin{cases}\mathfrak{p}_{10}^1\mathfrak{p}_{i-1,j}^{i+h-1} &\text{if $h>1$,}\\
    \mathfrak{p}_{0j}^1\mathfrak{p}_{i0}^i+\mathfrak{p}_{10}^1 \mathfrak{p}_{i-1,j}^i &\text{if $h=1$.}\end{cases}
  \end{equation}
  Assume that $h>1$. Applying~\eqref{inductivo para d mayor que uno} several times we obtain that
  $\mathfrak{p}_{ij}^{i+h}=\left(\mathfrak{p}^1_{10}\right)^i \mathfrak{p}_{0j}^h$. Since $\mathfrak{p}^1_{10}\ne 0$ and $\mathfrak{p}_{n-h,j}^n=0$,
  this equality implies that $\mathfrak{p}_{0j}^h=0$. Thus $\mathfrak{p}_{ij}^{i+d}=0$ for all $i$. Assume now that $h=1$. Then,
  by~\eqref{rho_ij^i+j} and~\eqref{inductivo para d mayor que uno},
  $$
    \mathfrak{p}_{ij}^{i+1}=\mathfrak{p}_{0j}^1\mathfrak{p}_{i0}^i
    +\mathfrak{p}_{10}^1\mathfrak{p}_{i-1,j}^i=\mathfrak{p}_{0j}^1\bigl(\mathfrak{p}_{10}^1\bigr)^i+ \mathfrak{p}_{10}^1 \mathfrak{p}_{i-1,j}^i.
  $$
  Using this it is easy to check by induction in $i$ that $\mathfrak{p}_{ij}^{i+1}=(i+1)\mathfrak{p}^1_{0j}\left(\mathfrak{p}^1_{10}\right)^i$. Since
  $\mathfrak{p}^1_{10}\ne 0$ and $\mathfrak{p}_{n-1,j}^n=0$, this equality implies that $\mathfrak{p}_{0j}^1=0$. Thus, $\mathfrak{p}_{ij}^{i+1}=0$
  for all $i$. This concludes the proof.
\end{proof}

\begin{proposition}\label{formula para d1j1}
  For all $i\ge 1$, we have
  $\mathfrak{p}_{i1}^i=i\mathfrak{p}_{11}^1\left(\mathfrak{p}^1_{10}\right)^{i-1}$ and
  $\mathfrak{d}_{i1}^i=i\mathfrak{d}_{11}^1\left(\mathfrak{d}^1_{10}\right)^{i-1}$.
\end{proposition}

\begin{proof}
  The case $i=1$ is trivial. Assume that $i>1$ and set $k=i$ and $j=l=1$ in~\eqref{eq: comultiplicative map}. By Proposition~\ref{lado derecho} the
  terms with $(a,b)\ne (i-1,1)$ vanish. So,
  $$
    \mathfrak{p}_{i1}^i=\mathfrak{p}_{10}^1\mathfrak{p}_{i-1,1}^{i-1}+\mathfrak{p}_{11}^1\mathfrak{p}_{i-1,0}^{i-1}=
    \mathfrak{p}_{10}^1\mathfrak{p}_{i-1,1}^{i-1}+ \mathfrak{p}_{11}^1
    \left(\mathfrak{p}_{10}^1\right)^{i-1}=\mathfrak{p}_{10}^1\left(\mathfrak{p}_{i-1,1}^{i-1} +\mathfrak{p}_{11}^1
    \left(\mathfrak{p}_{10}^1\right)^{i-2}\right),
  $$
  where the second equality holds by~\eqref{rho_ij^i+j}. The result for $\mathfrak{p}_{i1}^i$ follows from this formula by induction on $i$. For
  $\mathfrak{d}_{i1}^i$ the result follows mutatis mutandis.
\end{proof}

\begin{remark}\label{resumen}
  By Proposition~\ref{lado derecho} and Remarks~\ref{primeras condiciones} and~\ref{no degenerado a izquierda}, we know
  that $\mathfrak{p}_{ij}^0=\mathfrak{d}_{ij}^0= \delta_{0,i+j}$, $\mathfrak{p}_{10}^1\ne 0$, $\mathfrak{d}_{10}^1\ne 0$,
  $\mathfrak{p}_{i0}^i=(\mathfrak{p}_{10}^1)^i$, $\mathfrak{d}_{i0}^i =(\mathfrak{d}_{10}^1)^i$ and $\mathfrak{p}_{ij}^k =\mathfrak{d}_{ij}^k=0$
  for all $k>i$.
\end{remark}

\section{Solutions for the dual of the truncated polynomial algebra}
\setcounter{equation}{0}
In this section we study left non-degenerate bijective set-theoretic type solutions of the braid equation for the dual $C$ of the truncated
polynomial algebra $K[y]/\langle y^n\rangle$. By Theorem~\ref{soluciones de tipo conjuntista=q-cycle coalgebras}, for this it suffices to study the
regular $q$-cycle coalgebras whose underlying coalgebra is $C$.

\smallskip

Let $(C,\mathfrak{p},\mathfrak{d})$ be a regular $q$-magma coalgebra. Computing items~1, 2 and~3 of Definition~\ref{q cycle coalgebra} at $a=x_i$,
$b=x_j$ and $c=x_k$, and using Proposition~\ref{lado derecho}, we obtain
\begin{align}
  &\sum_{a+b=j}\sum_{h=0}^i\sum_{l=0}^k\sum_{m=0}^h\mathfrak{p}_{ia}^h\mathfrak{d}_{kb}^l\mathfrak{p}_{hl}^m x_m=\sum_{c+d=k} \sum_{h=0}^i
  \sum_{l=0}^j \sum_{m=0}^h \mathfrak{p}_{ic}^h\mathfrak{p}_{jd}^l\mathfrak{p}_{hl}^m x_m,\label{primera}\\
  & \sum_{a+b=j}\sum_{h=0}^i\sum_{l=0}^k\sum_{m=0}^h\mathfrak{p}_{ia}^h\mathfrak{p}_{kb}^l\mathfrak{d}_{hl}^m x_m=\sum_{c+d=k} \sum_{h=0}^i
  \sum_{l=0}^j \sum_{m=0}^h \mathfrak{d}_{ic}^h\mathfrak{d}_{jd}^l\mathfrak{p}_{hl}^m x_m\label{segunda},\\
  &\sum_{a+b=j}\sum_{h=0}^i\sum_{l=0}^k\sum_{m=0}^h\mathfrak{d}_{ia}^h\mathfrak{d}_{kb}^l\mathfrak{d}_{hl}^m x_m=\sum_{c+d=k} \sum_{h=0}^i
  \sum_{l=0}^j \sum_{m=0}^h \mathfrak{d}_{ic}^h\mathfrak{p}_{jd}^l\mathfrak{d}_{hl}^m x_m.\label{tercera}
\end{align}

\begin{proposition}\label{reduccion a m igual a 1}
  The triple $(C,\mathfrak{p},\mathfrak{d})$ is a $q$-cycle coalgebra if and only if the following
  equalities are true for all $i,j,k$:
  \begin{align}
    &\sum_{a+b=j}\sum_{h=0}^i\sum_{l=0}^k\mathfrak{p}_{ia}^h\mathfrak{d}_{kb}^l\mathfrak{p}_{hl}^1=\sum_{c+d=k}\sum_{h=0}^i\sum_{l=0}^j
    \mathfrak{p}_{ic}^h \mathfrak{p}_{jd}^l\mathfrak{p}_{hl}^1,\label{primera en m igual a 1}\\
    &\sum_{a+b=j}\sum_{h=0}^i\sum_{l=0}^k\mathfrak{p}_{ia}^h\mathfrak{p}_{kb}^l\mathfrak{d}_{hl}^1=\sum_{c+d=k}\sum_{h=0}^i\sum_{l=0}^j
    \mathfrak{d}_{ic}^h \mathfrak{d}_{jd}^l\mathfrak{p}_{hl}^1,\label{segunda en m igual a 1}\\
    &\sum_{a+b=j}\sum_{h=0}^i\sum_{l=0}^k\mathfrak{d}_{ia}^h\mathfrak{p}_{kb}^l
    \mathfrak{d}_{hl}^1=\sum_{c+d=k}\sum_{h=0}^i\sum_{l=0}^j\mathfrak{d}_{ic}^h\mathfrak{d}_{jd}^l\mathfrak{d}_{hl}^1.\label{tercera en m igual a 1}
  \end{align}
\end{proposition}

\begin{proof} It suffices to show that~\eqref{primera} is equivalent to~\eqref{primera en m igual a 1},~\eqref{segunda} is equivalent
  to~\eqref{segunda en m igual a 1} and~\eqref{tercera} is equivalent to~\eqref{tercera en m igual a 1}. We only prove the first equivalence  since
  the other ones are similar. Clearly equality~\eqref{primera en m igual a 1} holds if and only if the coefficients of $x_1$ at the left and right
  sides of the equal sign in~\eqref{primera} coincide. Assume that $m>1$, the coefficients of $x_{m-1}$ coincide and~\eqref{primera en m igual a 1}
  holds. By~\eqref{eq: comultiplicative map}, the coefficient of $x_m$ on the left hand side of~\eqref{primera} reads
  \allowdisplaybreaks
  \begin{align*}
    \sum_{a+b=j}\sum_{h=0}^i\sum_{l=0}^k\mathfrak{p}_{ia}^h \mathfrak{d}_{kb}^l\mathfrak{p}_{hl}^m
    &=\sum_{a+b=j}\sum_{h=0}^i\sum_{l=0}^k\sum_{\substack{h_1+h_2=h\\ l_1+l_2=l}}\mathfrak{p}_{ia}^h\mathfrak{d}_{kb}^l\mathfrak{p}_{h_1
    l_1}^{m-1}\mathfrak{p}_{h_2 l_2}^1\\
    &=\sum_{\substack{a+b=j\\ 0\le h\le i\\ 0\le l\le k}}\sum_{\substack{h_1+h_2=h\\ l_1+l_2= l}}\raisebox{-1.7ex}{$\begin{pmatrix}\displaystyle
    \sum_{\substack{i_1+i_2=i\\ a_1+a_2=a}}\mathfrak{p}_{i_1 a_1}^{h_1}\mathfrak{p}_{i_2 a_2}^{h_2}\sum_{\substack{k_1+k_2=k\\ b_1+b_2=b}}
    \mathfrak{d}_{k_1 b_1}^{l_1}\mathfrak{d}_{k_2 b_2}^{l_2}\end{pmatrix}$}\mathfrak{p}_{h_1 l_1}^{m-1}\mathfrak{p}_{h_2 l_2}^1\\
    &=\sum_{\substack{i_1+i_2=i\\ j_1+j_2=j\\ k_1+k_2=k}}\sum_{\substack{a_1+b_1=j_1\\ 0\le h_1\le i_1\\ 0\le l_1\le k_1}}\mathfrak{p}_{i_1a_1}^{h_1}
    \mathfrak{d}_{k_1b_1}^{l_1}\mathfrak{p}_{h_1l_1}^{m-1}\sum_{\substack{a_2+b_2=j_2\\ 0\le h_2\le i_2\\ 0\le l_2\le k_2}}
    \mathfrak{p}_{i_2a_2}^{h_2} \mathfrak{d}_{k_2b_2}^{l_2}\mathfrak{p}_{h_2l_2}^1.
  \end{align*}
  On the other hand, again by~\eqref{eq: comultiplicative map}, the coefficient of $x_m$ on the right hand side of~\eqref{primera} reads
  \allowdisplaybreaks
  \begin{align*}
    \sum_{c+d=k}\sum_{h=0}^i\sum_{l=0}^j\mathfrak{p}_{ic}^h \mathfrak{p}_{jd}^l\mathfrak{p}_{hl}^m
    &=\sum_{c+d=k}\sum_{h=0}^i\sum_{l=0}^j\sum_{\substack{h_1+h_2=h\\ l_1+l_2=l}}\mathfrak{p}_{ic}^h\mathfrak{p}_{jd}^l\mathfrak{p}_{h_1
    l_1}^{m-1}\mathfrak{p}_{h_2 l_2}^1 \\
    &=\sum_{\substack{c+d=k\\ 0\le h\le i\\ 0\le l\le j}}\sum_{\substack{h_1+h_2=h\\ l_1+l_2=l}} \raisebox{-1.7ex}{$\begin{pmatrix}
    \displaystyle\sum_{\substack{i_1+i_2=i\\ c_1+c_2=c}}\mathfrak{p}_{i_1 c_1}^{h_1}\mathfrak{p}_{i_2 c_2}^{h_2}\sum_{\substack{j_1+j_2=j\\
    d_1+d_2=d}} \mathfrak{p}_{j_1 d_1}^{l_1}\mathfrak{p}_{j_2 d_2}^{l_2}\end{pmatrix} $}\mathfrak{p}_{h_1 l_1}^{m-1}\mathfrak{p}_{h_2 l_2}^1 \\
    &=\sum_{\substack{i_1+i_2=i\\ j_1+j_2=j\\ k_1+k_2=k}}\sum_{\substack{c_1+d_1=k_1\\ 0\le h_1\le i_1\\ 0\le l_1\le j_1}}\mathfrak{p}_{i_1c_1}^{h_1}
    \mathfrak{p}_{j_1d_1}^{l_1}\mathfrak{p}_{h_1l_1}^{m-1}\sum_{\substack{c_2+d_2=k_2\\ 0\le h_2\le i_2\\ 0\le l_2\le j_2}}
    \mathfrak{p}_{i_2c_2}^{h_2}
    \mathfrak{p}_{j_2d_2}^{l_2}\mathfrak{p}_{h_2l_2}^1.
  \end{align*}
  Hence, by the inductive hypothesis, the coefficients of $x_m$ at the left hand side and the right hand side of~\eqref{primera} coincide.
\end{proof}

\begin{proposition}\label{prop casos principales}
  Let $(C,\mathfrak{p},\mathfrak{d})$ be a regular $q$-cycle coalgebra. Then, one of the following cases necessarily happen:
  $\mathfrak{d}_{10}^1=\mathfrak{p}_{10}^1=1$ or $\mathfrak{d}_{11}^1=\mathfrak{p}_{11}^1=0$.
\end{proposition}

\begin{proof}
  By Remark~\ref{resumen} and equality~\eqref{primera en m igual a 1} with $k=0$ and $i=j=1$, we have $\mathfrak{p}_{11}^1\mathfrak{p}_{10}^1=
  \mathfrak{p}_{10}^1\mathfrak{p}_{10}^1 \mathfrak{p}_{11}^1$. Since $\mathfrak{p}_{10}^1\ne 0$, this implies that $\mathfrak{p}_{11}^1=
  \mathfrak{p}_{10}^1\mathfrak{p}_{11}^1$. Similarly, using equality~\eqref{tercera en m igual a 1} with $k=0$ and $i=j=1$, we obtain that
  $\mathfrak{d}_{11}^1=\mathfrak{d}_{10}^1\mathfrak{d}_{11}^1$; while using equality~\eqref{primera en m igual a 1} with $j=0$ and $i=k=1$ we obtain
  that $\mathfrak{p}_{11}^1=\mathfrak{d}_{10}^1\mathfrak{p}_{11}^1$. Finally from equality~\eqref{tercera en m igual a 1} with $j=0$ and $i=k=1$, it
  follows that $\mathfrak{d}_{11}^1=\mathfrak{p}_{10}^1\mathfrak{d}_{11}^1$. The statement follows immediately from these facts.
\end{proof}

\begin{proposition}\label{d111 igual a p111}
  Let $(C,\mathfrak{p},\mathfrak{d})$ be a regular $q$-cycle coalgebra. Then
  $\mathfrak{p}_{11}^1=\mathfrak{d}_{11}^1$.
\end{proposition}

\begin{proof}
  By Proposition~\ref{prop casos principales}, we can assume that $\mathfrak{p}_{10}^1=\mathfrak{d}_{10}^1=1$.
  Equality~\eqref{primera en m igual a 1} with $i=j=k=1$ and Remark~\ref{resumen}, gives
  \begin{equation}\label{d111 casi igual a p111}
    0=2\mathfrak{p}_{10}^1\mathfrak{p}_{11}^1\mathfrak{p}_{11}^1-\mathfrak{p}_{10}^1\mathfrak{d}_{11}^1\mathfrak{p}_{11}^1-\mathfrak{p}_{11}^1
    \mathfrak{d}_{10}^1\mathfrak{p}_{11}^1= \mathfrak{p}_{11}^1(\mathfrak{p}_{11}^1-\mathfrak{d}_{11}^1).
  \end{equation}
  A similar computation using equality~\eqref{tercera en m igual a 1} with $i=j=k=1$, gives
  $$
    0=2\mathfrak{d}_{10}^1\mathfrak{d}_{11}^1\mathfrak{d}_{11}^1-\mathfrak{d}_{10}^1\mathfrak{p}_{11}^1\mathfrak{d}_{11}^1-\mathfrak{d}_{11}^1
    \mathfrak{p}_{10}^1\mathfrak{d}_{11}^1 =\mathfrak{d}_{11}^1(\mathfrak{d}_{11}^1-\mathfrak{p}_{11}^1),
  $$
  which, combined with~\eqref{d111 casi igual a p111}, yields the assertion.
\end{proof}

\begin{remark}
  Let $(C,\mathfrak{p},\mathfrak{d})$ be a regular $q$-cycle coalgebra. By Remark~\ref{resumen} and equality~\eqref{primera} with $k=0$ and $m=i$, we
  have $\sum_{l=0}^j\mathfrak{p}_{i0}^i\mathfrak{p}_{j0}^l\mathfrak{p}_{il}^i=\mathfrak{p}_{ij}^i\mathfrak{p}_{i0}^i$, and the same argument
  using~\eqref{tercera} instead~\eqref{primera}, gives $\sum_{l=0}^j\mathfrak{d}_{i0}^i\mathfrak{d}_{j0}^l\mathfrak{d}_{il}^i=
  \mathfrak{d}_{ij}^i\mathfrak{d}_{i0}^i$.	Since $\mathfrak{p}_{i0}^i\ne 0$ and $\mathfrak{d}_{i0}^{0}\ne 0$, we conclude that
  \begin{equation}\label{desarrollo para p0jk}
    \sum_{l=0}^j \mathfrak{p}_{j0}^l\mathfrak{p}_{il}^i=\mathfrak{p}_{ij}^i\qquad\text{and}\qquad\sum_{l=0}^j
    \mathfrak{d}_{j0}^l\mathfrak{d}_{il}^i=\mathfrak{d}_{ij}^i.
  \end{equation}
\end{remark}

\begin{remark}\label{equivalencia de q brazas}
  Let $\lambda\in K^{\times}$  and let $(C,\mathfrak{p},\mathfrak{d})$ be a regular $q$-magma coalgebra. A direct computation proves that
  $(C,\ov{\mathfrak{p}},\ov{\mathfrak{d}})$, where $\ov{\mathfrak{p}}$ and $\ov{\mathfrak{d}}$ are given by
  $\ov{\mathfrak{p}}_{ij}^k\coloneqq \lambda^{k-i-j} \mathfrak{p}_{ij}^k$ and $\ov{\mathfrak{d}}_{ij}^k\coloneqq\lambda^{k-i-j}\mathfrak{d}_{ij}^k$,
   is a regular $q$-magma coalgebra, and that the coalgebra automorphism $f_{\lambda}\colon C\to C$, given by
   $f_{\lambda}(x_i)\coloneqq\lambda^i x_i$,
   is a $q$-magma coalgebra isomorphism from $(C,\mathfrak{p},\mathfrak{d})$ to $(C,\ov{\mathfrak{p}},\ov{\mathfrak{d}})$.
\end{remark}

In the following  sections we will consider different values of the coefficients of $\mathfrak{p}$ and $\mathfrak{d}$ and classify these cases.

\section{Construction of standard cycle coalgebras}
\setcounter{equation}{0}
Let $K$ be an algebraically closed field of characteristic~$0$ and let $C$ be the dual coalgebra of the algebra $K[y]/\langle y^n\rangle$, where
$n\ge 2$. In this section  we will construct a family of cycle coalgebras $(C,\mathfrak{p})$ which we call standard cycle coalgebras. For each
$i,j,k\in \mathds{N}_0$ we consider the equation
\begin{equation}\tag{$E_{ijk}$}\label{E_{ijk}}
  \sum_{a+b=j}\sum_{h=0}^i \sum_{l=0}^k X_{ia}^h X_{kb}^l X_{hl}^1 = \sum_{c+d=k} \sum_{h=0}^i \sum_{l=0}^j X_{ic}^h X_{jd}^l X_{hl}^1.
\end{equation}
Given a family $\{\mathfrak{p}_{uv}^w\}_{u,v,w\in \mathds{N}_0}$ of elements of $K$ we set
$$
  R(i,j,k)\coloneqq \sum_{a+b=j}\sum_{h=0}^i\sum_{l=0}^k\mathfrak{p}^h_{ia}\mathfrak{p}^l_{kb}\mathfrak{p}^1_{hl}\quad\text{for each $i,j,k\in
  \mathds{N}_0$.}
$$
Clearly $(\mathfrak{p}_{uv}^w)_{u,v,w\in \mathds{N}_0}$ satisfies~\eqref{E_{ijk}} if and only if $R(i,j,k) = R(i,k,j)$.

\smallskip

Given $v_0>0$ and a formal series $f=1+\sum_{v\ge v_0} p_v x^v \in K[[x]]$, with $p_{v_0} = 1$, we will construct a family
$\{\mathfrak{p}_{uv}^w\}_{u,v,w\in \mathds{N}_0}$ with $\mathfrak{p}_{1v}^1=p_v$ for all $v\ge v_0$, satisfying~\eqref{E_{ijk}} for all $i,j,k$.
We will define the $\mathfrak{p}_{uv}^w$ in various steps. First we set $\mathfrak{p}_{uv}^0=\delta_{0,u+v}$. Then we define
$$
  g\coloneqq \frac{f(f^{v_0}-1)}{f'}\in K[[x]].
$$
So the equality
\begin{equation}\label{relation f and gi}
  gf'=f(f^{v_0}-1)
\end{equation}
is satisfied. Note that $g\in x+x^2 K[[x]]$. We now define the series $G=\sum_{v\ge 0} g_v(x) y^v\in K[[x]][[y]]$ by requiring that $g_0(x)=x$ and
that its formal partial derivatives $G_x$ and $G_y$ satisfy the identity
\begin{equation}\label{definition of G}
  G_y(x,y) = g(x) f'(y) G_x(x,y)-\ov{f}(y) G_y(x,y),
\end{equation}
where $\ov{f}\coloneqq f-1$. Note that~\eqref{definition of G} is equivalent to
\begin{equation}\label{definition of G por grado}
  (v+1)g_{v+1}= g \sum_{l=0}^{v-v_0+1} (v-l+1)p_{v-l+1} g_l'-\sum_{l=1}^{v-v_0+1}p_{v-l+1}lg_l,
\end{equation}
which together with $g_0=x$ determines $G$. Now we define the coefficients $\mathfrak{p}_{uv}^1$ by
\begin{equation}\label{definition of los p ij}
  G(x,y)=\sum_{u,v} \mathfrak{p}_{uv}^1 x^u y^v.
\end{equation}
Note that $\mathfrak{p}_{uv}^1$ is the coefficient of $x^u$ in $g_v$. Finally we recursively define
\begin{equation}\label{inductivo coproducto1}
  \mathfrak{p}_{uv}^w=\sum_{\substack{u_1+u_2=u \\ v_1+v_2=v}} \mathfrak{p}_{u_1 v_1}^1 \mathfrak{p}_{u_2 v_2}^{w-1}\qquad \text{for $w>1$.}
\end{equation}

\begin{lemma}\label{lema propiedades de G}
  The following equalities hold:
  \begin{enumerate}[itemsep=0.7ex, topsep=1.0ex, label={\arabic*.}]
    \item $g_v=0$ for $0<v<v_0$,

    \item $g_{v_0}=g$,

    \item $G(0,y)=0$,

    \item $G_x(0,y)=f(y)$,

    \item $\mathfrak{p}_{1v}^1=p_v$ for $v\ge v_0$,

    \item $\mathfrak{p}_{uv}^w=0$ for $u<w$.
  \end{enumerate}
\end{lemma}

\begin{proof}
  Items~1 and~2 follow from equality~\eqref{definition of G por grado}. Since $g(0)=0$, item~3 also follows from
  equality~\eqref{definition of G por grado} by an evident inductive argument. Set $\hat{f}(y):=G_x(0,y)$. If we take the partial derivative with
  respect to $x$ of~\eqref{definition of G} and evaluate at $x=0$, we obtain
  $$
    G_{xy}(0,y)=g'(0)f'(y)G_x(0,y)+g(0)f'(y)G_{xx}(0,y)-\ov{f}(y) G_{xy}(0,y).
  $$
  Since $g(0)=0$ and $g'(0)=1$, we have $f(y)\hat{f}'(y)=f'(y)\hat{f}(y)$, which implies that $\hat{f}(y)=f(y)$, because $f(0)=\hat{f}(0)=1$. Thus
  item~4 is true. Since $G_x(0,y)=\sum_{v\ge 0} \mathfrak{p}_{1v}^1 y^v$, item~5 follows from item~4. Finally, using that $\mathfrak{p}_{0v}^1=0$ for
  all $v$ (by item~3) and~\eqref{inductivo coproducto1}, we obtain item~6.
\end{proof}

One verifies directly using Lemma~\ref{lema propiedades de G}(6), that for $i=0$ both sides of~\eqref{E_{ijk}} vanish.

\begin{theorem}\label{teorema principal}
  The coefficients defined above Lemma~\ref{lema propiedades de G} satisfy the equations~\eqref{E_{ijk}}.
\end{theorem}

The proof of the theorem is postponed until the end of this section.

\begin{corollary}\label{existe q cycle}
  Given $n\ge 2$, $1\le v_0< n$ and coefficients $\{p_v\}_{v_0\le v<n}$ in $K$, with $p_{v_0}=1$, there exists a cycle coalgebra such that
  $\mathfrak{p}_{1v}^1 = \delta_{0v}$ for $v<v_0$, and $\mathfrak{p}_{1v}^1 = p_v$ for $v_0\le v<n$.
\end{corollary}

\begin{proof}
  We set $f\coloneqq 1+\sum_{v=v_0}^{n-1} p_v x^v\in K[x]\subset K[[x]]$, and define the coefficients $\mathfrak{p}_{uv}^w$ as above. By
  Theorem~\ref{teorema principal}, the equalities~\eqref{E_{ijk}} are satisfied for all $i,j,k<n$. Thus $(C,\mathfrak{p})$ is a cycle co\-algebra.
\end{proof}

\begin{definition}\label{SIQ def}
  The cycle coalgebra defined in Corollary~\ref{existe q cycle} is called the {\em standard cycle coalgebra associated with $f$} and denoted by
  $\scc(f)$.
\end{definition}

\subsection{Differential operators in $K[[x,y]]$}\label{Differential operators in K[[x,y]]}

We will consider the double series $G\in K[[x,y]]$ defined above as a series in $K[[x]][[y]]$. Recall that $g_0=x$. Moreover by items~1, 2 and~3 of
Lemma~\ref{lema propiedades de G} we know that $g_v=0$ for $0<v<v_0$, $g_{v_0}=g$ and $x|g_v$ for all $v$ (i.e., $x|G)$. In particular
\begin{equation}\label{calculo de ovG}
  G = x+\ov{G} \quad \text{ with }\quad \ov{G}=g(x) y^{v_0}+g_{v_0+1}(x) y^{v_0+1}+g_{v_0+2}(x) y^{v_0+2}+\cdots.
\end{equation}
Combining the equalities~\eqref{relation f and gi} and~\eqref{definition of G}, we obtain
\begin{equation}\label{relation G_y with G_x1}
  g(y) G_y(x,y)=(f^{v_0}(y)-1)g(x)G_x(x,y).
\end{equation}
Since $x\mid G$, for each $h = h_0+h_1x+h_2x^2+\cdots \in K[[x]]$, the formal series
$$
  h(G)\coloneqq h_0+h_1 G+h_2 G^2+h_3 G^3+\cdots\in K[[x]][[y]]
$$
is well defined. We define the operators $\partial_x^v\colon K[[x]] \to K[[x]]$, which will simplify the expressions in $R(i,j,k)$, implicitly by
\begin{equation}\label{definicion de partial1}
  h(G)=\sum_{v\ge 0} \partial_x^vh(x) y^v\qquad\text{for all $h\in K[[x]]$}.
\end{equation}
For each $P\in (K[[x]][[y]])$ we write $P = P_0(x)+ P_1(x)y + P_2(x)y^2+\cdots$. In particular
$$
  h(G) = h(G)_0(x) + h(G)_1(x)y + h(G)_2(x)y^2 +\cdots
$$
Clearly
\begin{equation}\label{valores particulares de partial}
  \partial_x^0 h = h,\qquad \partial_x^v h = 0\quad\text{for $0<v<v_0$}\qquad\text{and}\qquad \partial_x^vx=g_v\quad\text{for all $v$.}
\end{equation}
Since $G = x+\ov{G}$, we have
$
(G^j)_v = \sum_{k=0}^j \binom{j}{k} x^{j-k}(\ov{G}^k)_v.
$
Hence, if $v>0$, then
$$
  h(G)_v = \sum_{j\ge 1} \sum_{k=1}^j h_j \binom{j}{k} x^{j-k}(\ov{G}^k)_v = \sum_{k\ge 1} \Biggl(\sum_{j=k}^{\infty} h_j \binom{j}{k}
  x^{j-k}\Biggr)(\ov{G}^k)_v = \sum_{k\ge 1} \frac{1}{k!} (\ov{G}^k)_v h^{(k)},
$$
where $h^{(k)}$ denotes the $k$-th formal derivative of $h$. This implies
\begin{equation}\label{definicion alternativa partial1}
  \partial_x^v h = h(G)_v =\sum_{k\ge 1} \frac{1}{k!} (\ov{G}^k)_v h^{(k)} =\sum_{k=1}^v \frac{1}{k!} (\ov{G}^k)_v h^{(k)},
\end{equation}
since $(\ov{G}^k)_v=0$ for $k>v$. Combining this with~\eqref{calculo de ovG}, we obtain $\partial_x^{v_0}h=g h'$. Consequently, $\partial_x^{v_0}$
is a derivation and $\partial_x^{v_0} f=g f'=f(f^{v_0}-1)$, where the last equality follows from~\eqref{relation f and gi}.

\medskip

For each $P\in K[[x,y]]$, we let $P_{uv}$ denote the coefficient of $x^uy^v$ in $P$.

\begin{lemma}\label{p como binomial1}
  Let $v,d\in \mathds{N}_0$. If $v\ne 0$ or $d\ne 0$, then $\mathfrak{p}_{d+w,v}^w=\sum_{i=1}^w \binom{w}{i} (\ov{G}^i)_{d+i,v}$ for all $w\ge 1$.
\end{lemma}

\begin{proof}
  We will proceed by induction on $w$. By \eqref{definition of los p ij} and~\eqref{calculo de ovG}, we have
  $\mathfrak{p}_{d+1,v}^1= \ov{G}_{d+1,v}$, as desired. Assume that the formula holds for $w-1$. Then, by~\eqref{inductivo coproducto1},
  $$
    \mathfrak{p}_{d+w,v}^w=\sum_{\substack{0\le i\le v\\ 0\le e \le d+w}}\mathfrak{p}_{ei}^1 \mathfrak{p}_{d+w-e,v-i}^{w-1}.
  $$
  But by Lemma~\ref{lema propiedades de G}(6), we know that $\mathfrak{p}_{ei}^1=0$ for $e<1$ and $\mathfrak{p}_{d+w-e,v-i}^{w-1}=0$ if $d+w-e<w-1$
  (or equivalently, if $e>d+1$), and so, setting $c\coloneqq e-1$, we have
  $$
    \mathfrak{p}_{d+w,v}^w=\sum_{i=0}^v\sum_{c=0}^d\mathfrak{p}_{c+1,i}^1 \mathfrak{p}_{d-c+w-1,v-i}^{w-1} = \mathfrak{p}_{d+1,v}^1 +
    \mathfrak{p}_{d+w-1,v}^{w-1}+ \sum_{i=1}^{v-1}\sum_{c=0}^d\mathfrak{p}_{c+1,i}^1 \mathfrak{p}_{d-c+w-1,v-i}^{w-1},
  $$
  where the last equality holds since $\mathfrak{p}_{10}^1 =1$, $\mathfrak{p}_{c+1,0}^1 =0$ when $c>0$, $\mathfrak{p}_{w-1,0}^{w-1} = 1$ and
  $\mathfrak{p}_{d-c+w-1,0}^{w-1} = 0$ when $c<d$ (take into account that $xy\mid \ov{G}$). By the inductive hypothesis we have
  $$
    \mathfrak{p}_{d+1,v}^1 = \ov{G}_{d+1,v},\qquad \mathfrak{p}_{d+w-1,v}^{w-1} = \sum_{i=1}^{w-1} \binom{w-1}{i} (\ov{G}^i)_{d+i,v} =
    \sum_{l=0}^{w-2} \binom{w-1}{l+1} (\ov{G}^{l+1})_{d+l+1,v},
  $$
  and
  \begin{align*}
    \sum_{i=1}^{v-1}\sum_{c=0}^d \mathfrak{p}_{c+1,i}^1 \mathfrak{p}_{d-c+w-1,v-i}^{w-1}&= \sum_{i=1}^{v-1}\sum_{c=0}^d \ov{G}_{c+1,i}
    \sum_{l=1}^{w-1} \binom{w-1}{l} (\ov{G}^l)_{d-c+l,v-i}\\
    & = \sum_{l=1}^{w-1}\binom{w-1}{l} \sum_{i=1}^{v-1} \sum_{c=0}^d \ov{G}_{c+1,i}(\ov{G}^l)_{d-c+l,v-i}\\
    &=\sum_{l=1}^{w-1} \binom{w-1}{l}(\ov{G}^{l+1})_{d+l+1,v},
  \end{align*}
  where the last equality follow from the fact that $xy|\ov{G}$. Consequently,
  \begin{align*}
    \mathfrak{p}_{d+w,v}^w &= (\ov{G}^w)_{d+w,v} +\sum_{l=0}^{w-2} \left(\binom{w-1}{l}+ \binom{w-1}{l+1} \right)(\ov{G}^{l+1})_{d+l+1,v}\\
    & = (\ov{G}^w)_{d+w,v} +\sum_{l=0}^{w-2} \binom{w}{l+1}(\ov{G}^{l+1})_{d+l+1,v}\\
    &=\sum_{i=1}^w\binom{w}{i} (\ov{G}^i)_{d+i,v},
  \end{align*}
  as desired.
\end{proof}

The following proposition explains how the use of $\partial_x^i$ will simplify the expressions in $R(i,j,k)$. For each $q\in K[[x]]$, we let $q_i$
denote the coefficient of $x^i$ in $q$.

\begin{proposition}\label{suma es operador diferencial1}
  For all $u\ge 1$, $v\ge 0$ and $\ell =\sum_{i\ge 0}\ell_i x^i\in K[[x]]$, we have $\sum_{h=1}^u\mathfrak{p}_{uv}^h \ell_h=(\partial_x^v \ell)_u$.
\end{proposition}

\begin{proof}
  On one hand, by equality~\eqref{definicion alternativa partial1},
  \begin{equation*}
    (\partial_x^v \ell)_u= \sum_{k=1}^v \frac{1}{k!} \bigl((\ov{G}^k)_v\ell^{(k)}\bigr)_u = \sum_{k=1}^v \frac{1}{k!}
    \sum_{r=0}^u (\ov{G}^k )_{u-r,v}(\ell^{(k)})_r = \sum_{k=1}^v \frac{1}{k!} \sum_{r=0}^{u-k} (\ov{G}^k )_{u-r,v} (\ell^{(k)})_r,
  \end{equation*}
  where the last equality holds since $(\ov{G}^k)_{u-r,v}=0$ for $r>u-k$, because $x|\ov{G}$. On the other hand, by Lemma~\ref{p como binomial1},
  $$
    \sum_{h=1}^u\mathfrak{p}_{uv}^h \ell_h = \sum_{h=1}^u \sum_{k=1}^h\binom{h}{k} (\ov{G}^k)_{u-h+k,v} \ell_h = \sum_{k=1}^u \sum_{h=k}^u
    \binom{h}{k} (\ov{G}^k)_{u-h+k,v} \ell_h.
  $$
  Therefore
  $$
    \sum_{h=1}^u\mathfrak{p}_{uv}^h \ell_h = \sum_{k=1}^u \frac{1}{k!} \sum_{r=0}^{u-k} (\ov{G}^k)_{u-r,v}(\ell^{(k)})_r = \sum_{k=1}^{\min(u,v)}
    \frac{1}{k!} \sum_{r=0}^{u-k} (\ov{G}^k)_{u-r,v}(\ell^{(k)})_r = (\partial_x^v \ell)_u,
  $$
  where the second equality holds since $(\ov{G}^k)_{u-r,v}=0$ for $k>v$, because $y|\ov{G}$; and the last equality, holds since the sum
  $\sum_{r=0}^{u-k} (\ov{G}^k )_{u-r,v} (\ell^{(k)})_r$ is empty when $k>u$.
\end{proof}

We can extend the operators $\partial_x^v\in \End_K(K[[x]])$ to operators on $K[[x,y]]=K[[x]][[y]]$ simply by acting on the coefficient
ring $K[[x]]$. We also define operators $\partial_y^u$ on $K[[x,y]]$ by
\begin{equation}\label{definition of partial y1}
  \partial_y^u H \coloneqq \sum_{i=1}^u \frac{1}{i!}(\ov{G}^i)_u(y) H^{(i)}_y,
\end{equation}
where $(\ov{G}^i)_u(y)$ is obtained replacing $x$ by $y$ in $(\ov{G}^i)_u$ and $H^{(i)}_y$ is the partial derivative of $H$ with respect to $y$,
iterated $i$ times. Note that $\partial_x^v \partial_y^u=\partial_y^u \partial_x^v$.

\begin{definition}\label{definition global partial}
  For $k\ge 0$ we define operators $\partial^k$ on $K[[x,y]]$, by $\partial^k \coloneqq \sum_{a+b=k} \partial_x^a \partial_y^b$.
\end{definition}
\noindent Note that for $k>0$, by Proposition~\ref{suma es operador diferencial1} we have
  \begin{equation}\label{Gjk}
    \left(\partial^j G\right)_k = \sum_{a+b=j} \left(\partial_y^b \left( \partial_x^a G\right) \right)_k = \sum_{a+b=j}\sum_{l=1}^k
    \mathfrak{p}_{kb}^l \left(\partial_x^a G\right)_l = \sum_{a+b=j}\sum_{l=1}^k \mathfrak{p}_{kb}^l \partial_x^a g_l.
  \end{equation}
The importance of $\partial^k$ lies in the following result.

\begin{proposition}\label{remark importante1}
  For all $i,k>0$ we have $(\partial^j G)_{ik}=R(i,j,k)$.
\end{proposition}

\begin{proof}
  On one hand by Proposition~\ref{suma es operador diferencial1}, we have
  $$
    R(i,j,k) = \sum_{a+b=j}\sum_{h=1}^i\sum_{l=0}^k\mathfrak{p}^h_{ia}\mathfrak{p}^l_{kb}\mathfrak{p}^1_{hl} = \sum_{a+b=j}\sum_{l=0}^k
    \mathfrak{p}^l_{kb} \sum_{h=1}^i \mathfrak{p}^h_{ia}\mathfrak{p}^1_{hl} = \sum_{a+b=j}\sum_{l=0}^k \mathfrak{p}^l_{kb} (\partial_x^a g_l)_i.
  $$
  The result follows combining this with the equality~\eqref{Gjk} at degree $i$, since $\mathfrak{p}_{kb}^0=0$ for all $b$.
\end{proof}

\begin{remark}\label{caso ijk cero}
  By identity~\eqref{calculo de ovG} we have $\mathfrak{p}_{u0}^1 = \delta_{u1}$, which by~\eqref{inductivo coproducto1} implies that
  $\mathfrak{p}_{u0}^w = \delta_{uw}$. A direct computation using this and that $\mathfrak{p}_{uv}^0 = \delta_{0,u+v}$ shows that
  $(\mathfrak{p}_{uv}^w)_{u,v,w\in \mathds{N}_0}$ satisfies~\eqref{E_{ijk}} when $i=0$, $j=0$ or $k=0$. So by Proposition~\ref{remark importante1},
  in order to prove Theorem~\ref{teorema principal}, it suffices to prove that
  \begin{equation}\label{falta demostrar1}
    \left(\partial^k G\right)_{ij}=\left(\partial^j G\right)_{ik}\qquad\text{for all $i,j,k>0$}.
  \end{equation}
\end{remark}

Our first goal is to prove that the operators $\partial_x^j$ commute with each other. For this we find another expression for $\ov{G}$.

\begin{proposition}\label{prop 3.9}
  Set $P_1(x)\coloneqq g(x)$ and define recursively
  \begin{equation}\label{recursivo Pi}
    (v+1)P_{v+1}(x)=g(x) P_v'(x)-vP_v(x),\quad \text{for $v\ge 1$.}
  \end{equation}
  Then $\ov{G}(x,y)=\sum_{v\ge 1} \ov{f}^v(y) P_v(x)$, where $\ov{f} = f-1$.
\end{proposition}

\begin{proof}
  Set $R\coloneqq \sum_{v\ge 1} \ov{f}^v(y) P_v(x)$. Since $R_y = \sum_{v\ge 1} v \ov{f}^{v-1}(y)\ov{f}'(y) P_v(x)$, we have
  $$
    \sum_{v\ge 1}(v+1)\ov{f}^v(y) P_{v+1}(x)\ov{f}'(y) = R_y-P_1(x)\ov{f}'(y)\quad\text{and}\quad \sum_{v\ge 1}v \ov{f}^v(y) P_v(x)\ov{f}'(y)=
    \ov{f}(y) R_y.
  $$
  Using this, the fact that $R_x = \sum_{v\ge 1} \ov{f}^v(y) P_v'(x)$ and equality~\eqref{recursivo Pi}, we obtain
  $$
    g(x) R_x \ov{f}'(y)= \sum_{v\ge 1}\ov{f}^v(y)\bigl((v+1)P_{v+1}(x)+vP_v(x)\bigr)\ov{f}'(y) = R_y-g(x)\ov{f}'(y)+ \ov{f}(y) R_y.
  $$
  In other words $R_y=g(x)\ov{f}'(y)- \ov{f}(y) R_y + g(x) R_x \ov{f}'(y)$. Since $\ov{f}(y)=y^{v_0} +\sum_{v>v_0} p_v y^v$, we have $R_v=0$ for
  $v<v_0$ and $R_{v_0}(x)=P_1(x)=\ov{G}_{v_0}(x)$. Combining this with the fact that, by equality~\eqref{definition of G}, we know that
  $\ov{G}_y=g(x)\ov{f}'(y)- \ov{f}(y) \ov{G}_y + g(x) \ov{G}_x \ov{f}'(y)$, we obtain $R=\ov{G}$, as desired.
\end{proof}

Set $P(x,y)\coloneqq \sum_{v\ge 1} P_v(x) y^v$. We define operators $\tilde{\partial}_x^v\in \End_K(K[[x]])$, by
$\tilde{\partial}_x^v h\coloneqq h(P+x)_v$.
Clearly $\tilde{\partial}_x^0=\ide$ and $\tilde{\partial}^v_x x = P_v$ for $v>0$. Moreover, arguing as in the proof
of~\eqref{definicion alternativa partial1} we obtain
\begin{equation}\label{definicion alternativa partial2}
  \tilde{\partial}_x^v h(x) = \sum_{k=1}^v \frac{1}{k!}(P^k)_v h^{(k)}(x)\qquad\text{for $v>0$.}
\end{equation}
Consequently $\tilde{\partial}_x^1 h = P_1h' = g h' = \partial_x^{v_0} h$.  We also consider the operators $\tilde{\partial}_y^v\in \End(K[[y]])$,
obtained replacing $x$ by $y$ in the definition (so, $\tilde{\partial}_y^v h$ is computed by replacing $x$ by $y$ in $\tilde{\partial}_x^v h(x)$).
We extend the operators $\tilde{\partial}_x^v$ and $\tilde{\partial}_y^v$ to operators on $K[[x,y]]=K[[x]][[y]]=K[[y]][[x]]$ by acting on the
coefficients ring $K[[x]]$ and $K[[y]]$, respectively. It is easy to check that $\tilde{\partial}_x^v \xcirc \tilde{\partial}_y^{v'}=
\tilde{\partial}_y^{v'}\xcirc \tilde{\partial}_x^v$ for all $v,v'$.

\begin{remark}\label{relation Py con Px}
  We claim that $g(x)(P+x)_x=(y+1)P_y$. In fact, from~\eqref{recursivo Pi} we obtain
  $$
    g(x)P_x=\sum_{v\ge 1}g(x)P_v'(x)y^v= \sum_{v\ge 1}(v+1)P_{v+1}(x)y^v+\sum_{v\ge 1}vP_v(x)y^v,
  $$
  and so $g(x)P_x=P_y+yP_y-P_1(x)=P_y+yP_y-g(x)$, as desired.
\end{remark}

\begin{lemma}
  The following equalities hold:
  \begin{align}
    & v \tilde{\partial}_x^v=\tilde{\partial}_x^1 \tilde{\partial}_x^{v-1}-(v-1) \tilde{\partial}_x^{v-1}\quad\text{for $v\ge 1$,}
    \label{partial tilde propiedad}\\
    &\tilde{\partial}_x^u \tilde{\partial}_x^v=\tilde{\partial}_x^v \tilde{\partial}_x^u\quad\text{for $u,v\ge 0$,}
    \label{conmutacion partial tilde j k}
    \shortintertext{and}
    &\bigl(\tilde{\partial}_x^u P\bigr)_v = \bigl(\tilde{\partial}_x^v P\bigr)_u\quad\text{for $u,v>0$.}\label{conmutacion partial tilde j P k}
  \end{align}
\end{lemma}

\begin{proof}
  We first prove~\eqref{partial tilde propiedad}. For $v=1$ this is trivial, so we assume $v>1$. Let $h\in K[[x]]$. Using the chain rule and the very
  definition of $\tilde{\partial}_x^v$, we obtain
  $$
    v \tilde{\partial}_x^v h = v h(P+x)_v=\bigl(h(P+x)_y\bigr)_{v-1}= \bigl(h'(P+x) P_y\bigr)_{v-1}.
  $$
  Hence, by Remark~\ref{relation Py con Px},
  $$
    v \tilde{\partial}_x^v h=\bigl(h'(P+x) g(x) (P+x)_x- h'(P+x)y P_y\bigr)_{v-1}= \bigl(\tilde{\partial}_x^1(h(P+x)) -y h(P+x)_y\bigr)_{v-1}.
  $$
  Consequently
  $$
    v \tilde{\partial}_x^v h = \tilde{\partial}_x^1\tilde{\partial}_x^{v-1}h -\bigl(h(P+x)_y\bigr)_{v-2} =
    \tilde{\partial}_x^1\tilde{\partial}_x^{v-1}h -(v-1)h(P+x)_{v-1}= \tilde{\partial}_x^1\tilde{\partial}_x^{v-1}h -(v-1)\tilde{\partial}_x^{v-1}h,
  $$
  which proves~\eqref{partial tilde propiedad}. An inductive argument using~\eqref{partial tilde propiedad} shows that the operators
  $\tilde{\partial}_x^k$ are polynomials in $\tilde{\partial}_x^1$, which implies that~\eqref{conmutacion partial tilde j k} holds. Using this we
  obtain
  $$
    \bigl(\tilde{\partial}_x^u P\bigr)_v=  \tilde{\partial}_x^u P_v = \tilde{\partial}_x^u\tilde{\partial}_x^v x=
    \tilde{\partial}_x^v\tilde{\partial}_x^u x = \tilde{\partial}_x^v P_u = \bigl(\tilde{\partial}_x^v P\bigr)_u\quad\text{for $u,v>0$,}
  $$
  as desired.
\end{proof}

\begin{lemma}\label{potencia de G barra}
  We have $\ov{G}^u(x,y)=\sum_{v\ge u} (P^u)_v(x) \ov{f}^v(y)$.
\end{lemma}

\begin{proof}
  By Proposition~\ref{prop 3.9} we have
  $\ov{G}(x,y)=\sum_{v\ge 1} P_v(x)\ov{f}^v(y)$. So, $\cramped{\ov{G}^u(x,y) = \sum_{v\ge 1} (P^u)_v(x)\ov{f}^v(y)}$, which concludes the proof since
  $(P^u)_v=0$ for $v<u$.
\end{proof}

\begin{proposition}\label{partial como partial tilde}
  Let $\partial^v_x$ be as at the beginning of Subsection~\ref{Differential operators in K[[x,y]]}. The equality $\partial^v_x=\sum_{u\ge 0}
  \bigl(\ov{f}^u\bigr)_v \tilde{\partial}_x^u$ holds on $K[[x,y]]$, for all $v\ge 0$.
\end{proposition}

\begin{proof}
  For $v=0$ this is trivially true, since $\partial^0_x=\tilde{\partial}_x^0=\ide$ (see~\eqref{valores particulares de partial} and the comment above
  identity~\eqref{definicion alternativa partial2}). Let $v>0$ and let $h\in K[[x]]$. Since $\bigl(\ov{f}^0\bigr)_v = 0$, by
  Lemma~\ref{potencia de G barra} and equalities~\eqref{definicion alternativa partial1} and~\eqref{definicion alternativa partial2}, we have
  $$
    \partial_x^v h  = \sum_{u\ge 1} \frac{1}{u!}\bigl(\ov{G}^u\bigr)_v h^{(u)} = \sum_{u\ge 1} \frac{1}{u!}\sum_{l\ge u} \bigl(\ov{f}^l\bigr)_v
    (P^u)_l  h^{(u)} = \sum_{l\ge 1} \bigl(\ov{f}^l\bigr)_v \left(\sum_{u=1}^l \frac{1}{u!}(P^u)_l h^{(u)}\right) = \sum_{l\ge 1}
    \bigl(\ov{f}^l\bigr)_v \tilde{\partial}_x^l h.
  $$
  So, the equality in the statement holds on $K[[x]]$, and thus, on $K[[x,y]]$.
\end{proof}

\begin{corollary}\label{partial conmuta con gi}
  For all $u,v\ge 0$ we have $\partial_x^u \partial_x^v= \partial_x^v \partial_x^u$ and $\bigl(\partial_x^u G\bigr)_v= \bigl(\partial_x^v G\bigr)_u$.
\end{corollary}

\begin{proof}
  The first equality follows directly from Proposition~\ref{partial como partial tilde} and equality~\eqref{conmutacion partial tilde j k}. Then,
  by~\eqref{valores particulares de partial},
  $$
    \bigl(\partial_x^u G\bigr)_v=  \partial_x^u  g_v(x)=\partial_x^u\partial_x^v(x)= \partial_x^v\partial_x^u(x)=  \partial_x^v  g_u(x)=
    \bigl(\partial_x^v G\bigl)_u,
  $$
  as desired.
\end{proof}

Note that by the second identity in Corollary~\ref{partial conmuta con gi}, we have
\begin{equation}\label{partial x k e i0 conmutan}
  \partial^d_x g(x)=\bigl(\partial_x^dG\bigr)_{v_0}=\bigl(\partial_x^{v_0}G\bigr)_d\quad \text{ for all $d\ge 0$.}
\end{equation}

\begin{corollary} \label{partial global en funcion de tilde partial}
	Let $\tilde{\partial}^u\coloneqq\sum_{l=0}^u \tilde{\partial}^l_x
	\tilde{\partial}^{u-l}_y$. The equalities $\partial^v_y=\sum_{u\ge 0} \ov{f}^u(y)_v \tilde{\partial}^u_y$ and $\partial^v=\sum_{u\ge 0}
    \ov{f}^u(y)_v \tilde{\partial}^u$ hold on $K[[x,y]]$.
\end{corollary}

\begin{proof}
	The first equality holds  by Proposition~\ref{partial como partial tilde}. By the same proposition, we have
	$$
	  \partial^v = \sum_{h=0}^v \partial^h_x\partial^{v-h}_y
	  = \sum_{h=0}^v \Biggl(\sum_{l\ge 0}\bigl(\ov{f}^l\bigr)_h\tilde{\partial}^l_x\Biggr)
	  \Biggl(\sum_{n\ge 0}\bigl(\ov{f}^n\bigr)_{v-h}\tilde{\partial}^n_y\Biggr)
	  =\sum_{l\ge 0} \sum_{n\ge 0} \sum_{h=0}^{v} \bigl(\ov{f}^l\bigr)_h
	  \bigl(\ov{f}^n\bigr)_{v-h}\tilde{\partial}^l_x \tilde{\partial}^n_y,
	$$
	where the last equality holds since the sums are finite. Consequently,
	$$
	  \partial^v = \sum_{l\ge 0} \sum_{n\ge 0} \bigl(\ov{f}^{l+n}\bigr)_v \tilde{\partial}^l_x
	  \tilde{\partial}^n_y= \sum_{u\ge 0}  \bigl(\ov{f}^u\bigr)_v \sum_{l= 0}^u \tilde{\partial}^l_x
	  \tilde{\partial}^{u-l}_y = \sum_{u\ge 0}  \bigl(\ov{f}^u\bigr)_v \tilde{\partial}^u,
	$$
	as desired.
\end{proof}

\begin{lemma}\label{partial tilde global propiedad}
	On $K[[x,y]]$, we have
	$v\tilde{\partial}^v=\tilde{\partial}^1\tilde{\partial}^{v-1} - (v-1)\tilde{\partial}^{v-1}$.
\end{lemma}

\begin{proof}
	Since the operators  $\tilde{\partial}^v_x$ commute one with each other and with the maps $\tilde{\partial}^{v'}_y$,
    by~\eqref{partial tilde propiedad} we have
	\begin{align*}
		\tilde{\partial}^{v-1}\tilde{\partial}^1 & =  \sum_{l=0}^{v-1}  \tilde{\partial}_y^{v-1-l} \tilde{\partial}_x^l \tilde{\partial}_x^1+
		\sum_{l=0}^{v-1} \tilde{\partial}_x^l \tilde{\partial}_y^{v-1-l} \tilde{\partial}_y^1\\
		& = \sum_{l=0}^{v-1}  \tilde{\partial}_y^{v-1-l} \bigl((l+1)\tilde{\partial}_x^{l+1} +l\tilde{\partial}_x^l\bigr)
		+  \sum_{l=0}^{v-1} \tilde{\partial}_x^l \bigl((v-l)\tilde{\partial}_y^{v-l}+ (v-1-l)\tilde{\partial}_y^{v-1-l}\bigr)\\
		& = \sum_{l=1}^v  l \tilde{\partial}_y^{v-l} \tilde{\partial}_x^l+\sum_{l=0}^{v-1}  (v-l)\tilde{\partial}_x^l\tilde{\partial}_y^{v-l}
		+ \sum_{l=0}^{v-1}  l \tilde{\partial}_y^{v-1-l} \tilde{\partial}_x^l+\sum_{l=0}^{v-1}
        (v-1-l)\tilde{\partial}_x^l\tilde{\partial}_y^{v-1-l}\\
		&= v\sum_{l=0}^v   \tilde{\partial}_y^{v-l} \tilde{\partial}_x^l+ (v-1)\sum_{l=0}^{v-1}\tilde{\partial}_x^l\tilde{\partial}_y^{v-1-l}\\
		&=v\tilde{\partial}^v + (v-1)\tilde{\partial}^{v-1},
	\end{align*}
	where the second equality follows from~\eqref{partial tilde propiedad}.
\end{proof}

\begin{remark}
	We do not need explicit formulas for $\tilde{\partial}^v_x$, $\tilde{\partial}^v_y$ and $\tilde{\partial}^v$, but we note for the record
	that
	$$
    	\tilde{\partial}^v_x=\binom{\tilde{\partial}^1_x}{v},\quad \tilde{\partial}^v_y=\binom{\tilde{\partial}^1_y}{v}
	    \quad\text{and}\quad \tilde{\partial}^v=\binom{\tilde{\partial}^1}{v},
	$$
	where we use generalized binomial coefficients given for example by $\binom{\tilde{\partial}^1_x}{v}\coloneqq
    \frac{1}{v!}\tilde{\partial}^1_x(\tilde{\partial}^1_x-1)\dots (\tilde{\partial}^1_x-v+1)$. In fact, $\binom{\tilde{\partial}^1_x}{1}=
    \tilde{\partial}^1_x$ and the maps $\binom{\tilde{\partial}^1_x}{v}$ satisfy the recursive relations~\eqref{partial tilde propiedad}, since
	$$
		(v+1)\binom{\tilde{\partial}^{1}_x}{v+1} = \frac{1}{v!}\tilde{\partial}^1_x(\tilde{\partial}^1_x-1)\dots
		(\tilde{\partial}^1_x-v+1)(\tilde{\partial}^1_x-v) =
        \tilde{\partial}^{1}_x\binom{\tilde{\partial}^{1}_x}{v}-v\binom{\tilde{\partial}^{1}_x}{v}.
	$$
	The same argument yields the expressions for $\tilde{\partial}^v_y$ and $\tilde{\partial}^v$.
\end{remark}

\begin{definition}
  We fix the series $F(x,y)\in K[[x,y]]$ setting $F(x,y)\coloneqq \sum_{i,j\ge 0} \frak{p}_{ij}^1 x^j y^i$.
\end{definition}
\noindent Note that $F(x,y)=G(y,x)$, $F_1=f(x)\in K[[x]]$, $F_0=0$ and
$$
  F_k=f_k(x)=\frak{p}_{k1}^1 x+\frak{p}_{k2}^1 x^2+\frak{p}_{k3}^1 x^3+\dots,\quad \text{for $k>1$.}
$$
We can write $F=\ov{F}+y$, with $\ov{F}(x,y)=\ov{G}(y,x)$, and so we have
\begin{equation}\label{transformar G en F1}
  \left( \ov{G}^l\right)_{jk}= \left( \ov{F}^l\right)_{kj}.
\end{equation}
The equality~\eqref{relation G_y with G_x1} now reads
\begin{equation}\label{partial x contra partial y de G1}
  \partial_y^{v_0} G= (f(y)^{v_0}-1) \partial_x^{v_0} G,
\end{equation}
which is equivalent to
\begin{equation}\label{partial x contra partial y de F1}
  \partial_x^{v_0} F= (f(x)^{v_0}-1) \partial_y^{v_0} F.
\end{equation}

\begin{proposition}\label{R se cumple para i0}
  For all $j,k$ we have $R(i,v_0,k)=R(i,k,v_0)$.
\end{proposition}

\begin{proof}
  For $i=0$ or $k=0$ this equality is true, so we assume $i,k>0$.
  By Proposition~\ref{suma es operador diferencial1} and~\eqref{partial x k e i0 conmutan}, and since $\frak{p}_{h,l}^{1}=0$ for $0<l<v_0$, we have
  \begin{align*}
     R(i,k,v_0) & = \sum_{c+d=k}\sum_{h=1}^{i}\sum_{l=1}^{v_0} \frak{p}_{i,c}^{h} \frak{p}_{v_0,d}^{l}
     \frak{p}_{h,l}^{1}=
     \sum_{c+d=k}\sum_{h=1}^{i} \frak{p}_{i,c}^{h} \frak{p}_{v_0,d}^{v_0}
     \frak{p}_{h,v_0}^{1}
     =  \sum_{c+d=k} \frak{p}_{v_0,d}^{v_0} \sum_{h=1}^{i} \frak{p}_{i,c}^{h} \frak{p}_{h,v_0}^{1}\\
     &=  \sum_{c+d=k} \frak{p}_{v_0,d}^{v_0} \left( \partial_x^c g(x)\right)_i
      =  \sum_{c+d=k} \frak{p}_{v_0,d}^{v_0} \left(\left( \partial_x^{v_0} G\right)_c\right)_i
     =  \sum_{c+d=k}\left( \left(f(y)^{v_0}\right)_d \left( \partial_x^{v_0} G\right)_c\right)_i\\
     &= \left( \left( f(y)^{v_0}  \partial_x^{v_0} G\right)_k\right)_i= \left( f(y)^{v_0}  \partial_x^{v_0} G\right)_{ik},
  \end{align*}
  where we use that by~\eqref{inductivo coproducto1} and items~(4) and~(6) of Lemma~\ref{lema propiedades de G}, we have
  \begin{equation}\label{expresion para pci0i0}
    \frak{p}_{v_0,d}^{v_0}=\left(f(y)^{v_0}\right)_d\in K,
  \end{equation}
   which is an element of the coefficient ring $K[[x]]$ of
  $(K[[x]])[[y]]$. Since $\partial^{v_0}= \partial^{v_0}_x+\partial^{v_0}_y$, from~\eqref{partial x contra partial y de G1}
  and Proposition~\ref{remark importante1} we obtain
  $$
    R(i,k,v_0)=\left( f(y)^{v_0}  \partial_x^{v_0} G\right)_{ik}=
    \left( \partial_x^{v_0} G + \partial_y^{v_0} G\right)_{ik}
    = \left( \partial^{v_0} G\right)_{ik}= R(i,v_0,k),
  $$
  as desired.
\end{proof}

In order to prove our main result, we will generalize the ideas of the proof of
Proposition~\ref{R se cumple para i0}. We first generalize the formula~\eqref{expresion para pci0i0}.
\begin{lemma} \label{expresion para pcjh}
  We have
  $$
    \frak{p}_{ja}^h=\left( F^h\right)_{aj}.
  $$
\end{lemma}

\begin{proof}
  If $a=0$, then $\frak{p}_{ja}^h=\delta_{jh}$, which coincides with $\left( F^h\right)_{0j}$,
  since $F=y f(x)+y^2 f_2(x)+\dots$, where $f(x)=1+x^{v_0}+p_{v_0+1}x^{v_0+1}+\dots$, and
  $x|f_i(x)$ for $i>1$. If $a>0$, then
  $$
    \left( F^h\right)_{aj}= \left( (\ov{F}+y)^h\right)_{aj}
    =\sum_{l=0}^{h} \binom{h}{l}\left( \ov{F}^l y^{h-l}\right)_{aj}
    =\sum_{l=0}^{h} \binom{h}{l}\left( \ov{F}^l\right)_{a,j-h+l}.
  $$
  But for $l=0$ we have
  $$
  \binom{h}{l}\left( \ov{F}^l\right)_{a,j-h+l}=(1)_{a,j-h}=0,
  $$
  since $a>0$, and the constant $1$ is seen as an element in $K[[x,y]]$.
  Hence, by~\eqref{transformar G en F1} and Lemma~\ref{p como binomial1},
  $$
    \left( F^h\right)_{aj}
    =\sum_{l=1}^{h} \binom{h}{l}\left( \ov{F}^l\right)_{a,j-h+l}
    =\sum_{l=1}^{h} \binom{h}{l}\left( \ov{G}^l\right)_{j-h+l,a}= \frak{p}_{ja}^h,
  $$
  as desired.
\end{proof}

Now we establish a convenient expression for
$\left( \partial^i G\right)_j$.

\begin{proposition}
  For $i,j>0$ we have
  \begin{equation}\label{Gij como F}
    \left( \partial^{i} G\right)_j=\sum_{h=1}^{j}\left(\left( F^h\right)_j(y) \partial_x^{h}G\right)_i,
  \end{equation}
  where we consider the series $(F^h)_j\in K[[x]]$ and replace $x$ by $y$, obtaining $(F^h)_j(y)$.
\end{proposition}

\begin{proof}
  By~\eqref{Gjk}, Lemma~\ref{expresion para pcjh} and Corollary~\ref{partial conmuta con gi}, we have
  $$
      \left(\partial^i G\right)_j
    = \sum_{a+b=i}\sum_{h=1}^{j} \frak{p}_{ja}^h \left(\partial_x^b G\right)_h
    = \sum_{h=1}^{j} \sum_{a+b=i} \left(\left( F^h\right)_j\right)_a \left(\partial_x^h G\right)_b
    = \sum_{h=1}^{j} \sum_{a+b=i} \left(\left( F^h\right)_j(y)\right)_a \left(\partial_x^h G\right)_b.
  $$
  Since
  $$
  \sum_{a+b=i} \left(\left( F^h\right)_j(y)\right)_a \left(\partial_x^h G\right)_b=
   \left( \left(F^h\right)_j (y)\partial_x^h G\right)_i,
  $$
  this concludes the proof.
\end{proof}

\begin{remark}\label{comparativo}
  It order to prove that $\left( \partial^{i} G\right)_j=\left( \partial^{j} G\right)_i$ for all $i,j>0$, it suffices to prove that
  \begin{equation}\label{Lo que falta para partial en G}
   \partial^{j} G =\sum_{h=1}^{j}\left(F^h\right)_j(y) \partial_x^{h}G,\quad \text{for all $j>0$,}
  \end{equation}
  or equivalently, that
  \begin{equation}\label{Lo que falta para partial en F}
   \partial^{j} F =\sum_{h=1}^{j}\left(F^h\right)_j \partial_y^{h}F,\quad \text{for all $j>0$.}
  \end{equation}
  By Corollary~\ref{partial global en funcion de tilde partial} the left hand side of this equality reads
  $$
    \partial^j F=\sum_{i\ge 1} \left( \ov{f}(y)^i \right)_j \tilde{\partial}^i F,
  $$
  and using the same corollary we also obtain
  \begin{align*}
    \sum_{h=1}^{j}\left(F^h\right)_j \partial_y^{h}F &
    = \sum_{h=1}^{j}\left(F^h\right)_j \sum_{i\ge 1} \left( \ov{f}(y)^i \right)_h \tilde{\partial}_y^{i}F
    =  \sum_{i\ge 1} \sum_{h=1}^{j} \left( \ov{f}(y)^i \right)_h \left(F^h\right)_j  \tilde{\partial}_y^{i}F\\
    & = \sum_{i\ge 1} \left(\sum_{h\ge 1} \left( \ov{f}(y)^i \right)_h \left(F^h\right)_j  \right) \tilde{\partial}_y^{i}F
    = \sum_{i\ge 1} \left(\ov{f}(F)^i \right)_j  \tilde{\partial}_y^{i}F,
  \end{align*}

  It follows that~\eqref{Lo que falta para partial en F} is equivalent to
  \begin{equation}\label{igualdad principal para las series de F}
    \sum_{i\ge 1} \left( \ov{f}(y)^i \right)_j \tilde{\partial}^i F
    = \sum_{i\ge 1} \left(\ov{f}(F)^i \right)_j  \tilde{\partial}_y^{i}F,\quad \text{for all $j>0$.}
  \end{equation}
\end{remark}

\subsection{Composition of series in $(K[[x]])[[y]]$}
In this subsection we will prove~\eqref{igualdad principal para las series de F} using composition of series in $(K[[x]])[[y]]$.
The first (technical) step is to write $\ov{f}(F)$ as $T(\ov{f}(y))$ for some series $T\in (K[[x]])[[y]]$
(see Proposition~\ref{ov f F como T de f ov}).
Then we prove that $\tilde{\partial}^k F=\sum_{h=1}^{k}(T^h)_k \tilde{\partial}^h_y F$ in
Proposition~\ref{tilde partial de F en funcion de tilde partial sub y}, which allows us to prove~\eqref{igualdad principal para las series de F}
in Corollary~\ref{prueba de la igualdad principal para F}.

\begin{lemma}
  There exists a unique series $Q_1(x)\in K[[x]]$ of the form $Q_1(x)=\sum_{i\ge 1} b_i x^i$ with $b_1=1$,
  such that
  \begin{equation}\label{def de Q1}
    g(x) Q_1'(x)=Q_1(x).
  \end{equation}
\end{lemma}

\begin{proof}
  The equality~\eqref{def de Q1} at degree $k$ yields
  $$
    b_k=\sum_{j=0}^{k} (g(x))_{k-j} (j+1)b_{j+1}.
  $$
  Since $(g(x))_0=0$, we obtain
  $$
    b_k=\sum_{j=1}^{k} (g(x))_{k-j+1} jb_{j}.
  $$
  This gives the recursive formula
  $$
    (k-1)b_k=-\sum_{j=1}^{k-1} (g(x))_{k-j+1} jb_{j},\quad \text{for $k\ge 2$,}
  $$
  and since $b_1=1$ is fixed, the series $Q_1(x)=\sum_{i\ge 1} b_i x^i$ is uniquely determined and clearly
  satisfies~\eqref{def de Q1}.
\end{proof}

Since $b_0=0$ and $b_1=1\ne 0$, there exists a compositional inverse of $Q_1(x)$ which we call $A$. So
$A(x)=\sum_{i\ge 1} a_i x^i$ and
$$
  \sum_{i\ge 1} a_i Q_1(x)^i=\sum_{i\ge 1} b_i A(x)^i=x.
$$
Clearly $a_1=1$.

\begin{proposition}
  We have
  \begin{equation}\label{G en terminos de f}
    G=\sum_{i\ge 1} a_i Q_1(x)^i f(y)^i.
  \end{equation}
\end{proposition}

\begin{proof}
  Set $R=\sum_{i\ge 1} a_i Q_1(x)^i f(y)^i$. Then
  $$
   f(y) R_y=\sum_{i\ge 1}i a_i Q_1(x)^i f(y)^i f'(y)
  $$
  and
  $$
   g(x)f'(y) R_x=\sum_{i\ge 1}i a_i Q_1(x)^{i-1}(g(x) Q_1'(x)) f(y)^i f'(y)=
   \sum_{i\ge 1}i a_i Q_1(x)^i f(y)^i f'(y),
  $$
  where the second equality follows from~\eqref{def de Q1}. Since
  $f(0)=1$, we also have
  $$
  R_0= R(x,0)= \sum_{i\ge 1} a_i Q_1(x)^i=x.
  $$
  But $G$ is the unique element in $(K[[x]])[[y]]$ satisfying
  $G_0=x$ and $f(y) G_y=g(x) f'(y) G_x$, and so we conclude that $G=R$.
\end{proof}

We can use~\eqref{G en terminos de f} in order to expand (again) $\ov{G}$ in powers of $\ov f$.
We have
$$
  \ov{G}+x= G=\sum_{i\ge 1} a_i Q_1(x)^i (\ov{f}(y)+1)^i=
  \sum_{i\ge 1} \sum_{j=0}^{i} \binom{i}{j} a_i Q_1(x)^i \ov{f}(y)^j.
$$
For $j=0$ we have
$$
  \sum_{i\ge 1} \binom{i}{j} a_i Q_1(x)^i \ov{f}(y)^j= \sum_{i\ge 1} a_i Q_1(x)^i=x,
$$
and so
$$
  \ov{G}= \sum_{i\ge 1} \sum_{j=1}^{i} \binom{i}{j} a_i Q_1(x)^i \ov{f}(y)^j=
  \sum_{j\ge 1} \left(\sum_{i\ge j} \binom{i}{j} a_i Q_1(x)^i\right) \ov{f}(y)^j.
$$
Since the expansion in powers of $\ov{f}(y)$ is unique, from Proposition~\ref{prop 3.9} we obtain
$$
  P_j(x)=\sum_{i\ge j} \binom{i}{j} a_i Q_1(x)^i.
$$
If we set $Q_i(x) \coloneqq a_i Q_1(x)^i$ and $Q=\sum_{i\ge 1} Q_i(x) y^i$, then $P$ is the binomial transform
of $Q$ in $(K[[x]])[[y]]$ and so $Q$ is the inverse binomial transform of $P$, i.e.,
$$
  Q_j(x)= \sum_{i\ge j} (-1)^{i-j}\binom{i}{j} P_i(x).
$$
We also have $g(x) Q_i'(x)=i Q_i(x)$ and so
\begin{equation}\label{derivadas parciales de Q}
  g(x) Q_x=y Q_y,
\end{equation}
since
$$
  g(x)Q_x=\sum_{i\ge 1} g(x)Q_i'(x) y^i=\sum_{i\ge 1} iQ_i(x) y^i=y Q_y.
$$

In general we can analyze the solutions of the differential equations
$$
  g(x) R'(x)=k R(x).
$$
\begin{proposition}\label{soluciones ecuacion diferencial}
  Let  $k\in\mathds{N}$.
  If $R(x)\in K[[x]]\setminus\{0\}$ satisfies $g(x) R'(x)=k R(x)$, then
  \begin{enumerate}
    \item $\ord(R(x))=k$ and
    \item there exists $\lambda\in K$ such that $R(x)=\lambda Q_1(x)^k$.
  \end{enumerate}
\end{proposition}

\begin{proof}
  Let $r=\ord (R)$.
  Then we have
  $$
    R(x)=\lambda_r x^r+\lambda_{r+1}x^{r+1}+\dots,\quad\text{with $\lambda_r\ne 0$,}
  $$
  and since $g(x)=x+\frak{p}_{v_0,2}^1 x^2+\dots,$ we obtain
  $$
    r\lambda_r x^r+(r\lambda_r \frak{p}_{v_0,2}^1+(r+1)\lambda_{r+1} )x^{r+1}+\dots=g(x) R'(x)=k R(x)=
    k\lambda_r x^r+k\lambda_{r+1}x^{r+1}+\dots.
  $$
  Since $\lambda\coloneqq \lambda_r\ne 0$, necessarily $r=k$, which proves item~(1).

  By definition $g(x) Q_1'(x)= Q_1(x)$, which implies
  $$
    g(x) (Q_1(x)^k)'= g(x) k Q_1(x)^{k-1}Q_1'(x)=k Q_1(x)^{k-1}(g(x) Q_1'(x))=k Q_1(x)^k.
  $$
  It follows that $\tilde R(x) \coloneqq R(x)-\lambda Q_1(x)^k$ satisfies
  $g(x) \tilde R'(x)=k \tilde R(x)$. Note that $\ord(R)=k$ and $\ord(\lambda Q_1(x)^k)=k$. Moreover,
  the lowest order coefficient of $R$ coincides with the lowest order coefficient of $\lambda Q_1(x)^k$.
  Hence, if $\tilde R\ne 0$, then $\ord(\tilde{R})>k$, which is impossible by item~(1), and so
  $\tilde R=0$, which proves item~(2).
\end{proof}

\begin{proposition}
  We have
  \begin{eqnarray}
    v_0 Q_1(x)^{v_0} &=& 1-\frac{1}{f(x)^{v_0}}, \label{Q1 en funcion de f}\\
    v_0 Q_1(x)^{v_0} &=& -\sum_{k\ge 1} \binom{-v_0}{k} \ov{f}(x)^k \label{Q1 binomial generalizada}\\
    f(A(x))^{v_0} &=& \frac{1}{1-v_0 x^{v_0}},    \label{f(A)}\\
    f(A(x)) &=& \sum_{k\ge 0} \binom{\frac{1}{v_0}+k-1}{k}\left( v_0 x^{v_0}\right)^k.\label{f(A) ralo}
  \end{eqnarray}
\end{proposition}

\begin{proof}
  We compute
  $$
    g(x)\left(1-\frac{1}{f(x)^{v_0}}\right)'=v_0 g(x)\frac{f'(x)}{f(x)^{v_0+1}}=
    v_0\frac{f(x)(f(x)^{v_0}-1)}{f(x)^{v_0+1}}=v_0 \left(1-\frac{1}{f(x)^{v_0}}\right).
  $$
  By Proposition~\ref{soluciones ecuacion diferencial} we know that
  $$
    1-\frac{1}{f(x)^{v_0}}=\lambda Q_1(x)^{v_0}.
  $$
  But
  $$
    \frac{1}{f(x)^{v_0}}=\left(\frac{1}{1+\ov{f}(x)}\right)^{v_0}=
    \left(1-\ov{f}(x)+\ov{f}(x)^2-\ov(f)(x)^3+\dots\right)^{v_0}=
    1-v_0 \ov{f}(x)+\ov{f}(x)^2 R_1(x),
  $$
  for some $R_1(x)\in K[[x]]$.  Since
  $$
    \ord(\ov{f}(x))=\ord(Q_1(x)^{v_0})=v_0\quad \text{and}\quad (\ov{f}(x))_{v_0}=(Q_1(x)^{v_0})_{v_0}=1,
  $$
  this shows that $\lambda=v_0$,  which proves~\eqref{Q1 en funcion de f}.
  The equality~\eqref{Q1 binomial generalizada} follows from the generalized binomial series.
  Now we replace $x$ by $A(x)$ in~\eqref{Q1 en funcion de f} and obtain
  $$
    v_0 x^{v_0}=1-\frac{1}{f(A(x))^{v_0}},
  $$
  since $A(x)$ is the compositional inverse of $Q_1(x)$. A straightforward computation
  yields~\eqref{f(A)}, from which~\eqref{f(A) ralo} follows, taking the generalized binomial series of the
  $v_0$-th root.
\end{proof}
By~\eqref{G en terminos de f} we have $F=A(Q_1(y)f(x))$. From~\eqref{f(A) ralo} it follows that
\begin{equation}\label{serie de f(F)}
  \ov{f}(F)=-1+ f(A(Q_1(y)f(x)))=
  \sum_{k\ge 1} \binom{\frac{1}{v_0}+k-1}{k}\left( v_0 Q_1(y)^{v_0}f(x)^{v_0}\right)^k.
\end{equation}
We define the following series in $(K[[x]])[[y]]$.
\begin{eqnarray*}
  S(y) &\coloneqq& 1-\frac{1}{(y+1)^{v_0}}= -\sum_{k\ge 1}\binom{-v_0}{k} y^k, \\
  V(y) &\coloneqq& \frac{1}{(1-y)^{1/v_0}}-1= \sum_{k\ge 1}\binom{\frac{1}{v_0}+k-1}{k} y^k, \\
  U(y) &\coloneqq& V(f(x)^{v_0}y)= \sum_{k\ge 1}\binom{\frac{1}{v_0}+k-1}{k} (f(x)^{v_0}y)^k, \\
  T(y) &\coloneqq& U\circ S(y)=V(f(x)^{v_0}S(y)).
\end{eqnarray*}
Note that $S(y),V(y)\in K[[y]]\subset (K[[x]])[[y]]$, and so we can consider formal derivatives
$$
  S'(y), V'(y)\in K[[x]]\subset (K[[x]])[[y]].
$$
\begin{proposition} \label{ov f F como T de f ov}
  We have
  \begin{eqnarray}
    v_0 Q_1(y)^{v_0} &=& S(\ov{f}(y)), \label{S propiedad} \\
    \ov{f}(F) &=& U(v_0 Q_1(y)^{v_0}), \label{U propiedad} \\
    \ov{f}(F) &=& T(\ov{f}(y)). \label{T propiedad}
  \end{eqnarray}
\end{proposition}

\begin{proof}
  The equality~\eqref{S propiedad} follows from~\eqref{Q1 binomial generalizada} and the definition of $S$,
  the equality~\eqref{U propiedad} follows from~\eqref{serie de f(F)} and the definition of $U$, and~\eqref{T propiedad}
  is a direct consequence of~\eqref{S propiedad} and~\eqref{U propiedad}.
\end{proof}

\begin{lemma}
  The series $T$ satisfies
  \begin{equation}\label{T derivado}
    (1+y)T_y=\tilde{\partial}^1_x T+f(x)^{v_0}(T+1).
  \end{equation}
\end{lemma}

\begin{proof}
  Write $W=W(x)=f(x)^{v_0}$. Then
  $$
    \tilde{\partial}^1_x(W)=g(x)(f(x)^{v_0})'=v_0 f(x)^{v_0-1}g(x)f'(x)=v_0 f(x)^{v_0-1} f(x)(f(x)^{v_0}-1)=v_0 W(W-1).
  $$
  We compute the formal derivative $V'(y)$ and obtain
  $$
    V'(y)=-\frac{1}{v_0}(1-y)^{-1/v_0-1}(-1)=\frac{1}{v_0}(1-y)^{-1/v_0-1}=\frac{V(y)+1}{v_0(1-y)}.
  $$
  Since $\tilde{\partial}^1_x$ is a derivation and $\tilde{\partial}^1_x(S(y))=0$, from the chain rule it follows that
  $$
    \tilde{\partial}^1_x T=V'(W S(y))S(y) \tilde{\partial}^1_x(W)=\frac{(V(W S(y))+1)}{v_0(1-W S(y))} S(y)v_0 W(W-1)
    = W(T+1)\frac{W S(y)-S(y)}{1-W S(y)}.
  $$
  From the chain rule it also follows that
  $$
    T_y=V'(W S(y))W S'(y),
  $$
  where the formal derivative $S'(y)$ is given by
  $$
    S'(y)=\frac{v_0}{(1+y)^{v_0+1}}=\frac{v_0(1-S(y))}{1+y}.
  $$
  Hence
  $$
    (1+y)T_y=\frac{(V(W S(y))+1)}{v_0(1-W S(y))}W v_0 (1-S(y))=W(T+1)\frac{1-S(y)}{1-W S(y)},
  $$
  and so
  $$
    (1+y)T_y-\tilde{\partial}^1_x T=W(T+1)\frac{1-S(y)-(W S(y)-S(y))}{1-W S(y)}=W(T+1),
  $$
  as desired.
\end{proof}

\begin{proposition} \label{tilde partial de F en funcion de tilde partial sub y}
  For all $k\ge 1$ we have
  \begin{equation}\label{eq tilde partial de F en funcion de tilde partial sub y}
    \tilde{\partial}^k F=\sum_{h=1}^{k}(T^h)_k \tilde{\partial}^h_y F.
  \end{equation}
\end{proposition}

\begin{proof}
  Since by definition $U_0=0$ and $S_0=0$, we have $T_0=0$ and from~\eqref{T derivado} at degree 0 we obtain $(T^1)_1=T_1=f(x)^{v_0}$.
  On the other hand, from~\eqref{partial x contra partial y de F1} we obtain
  $$
    \tilde{\partial}^1 F= \tilde{\partial}^1_x F+\tilde{\partial}^1_y F=f(x)^{v_0} \tilde{\partial}^1_y F
  $$
  since $\tilde{\partial}_x^1=\partial_x^{v_0}$ and $\tilde{\partial}_y^1=\partial_y^{v_0}$.
  So~\eqref{eq tilde partial de F en funcion de tilde partial sub y} holds for $k=1$.
  Consequently, it suffices to prove that the series $\sum_{h=1}^{k}(T^h)_k \tilde{\partial}^h_y F$ satisfy the same recursive relations as
  the series $\tilde{\partial}^k F$, given in Lemma~\ref{partial tilde global propiedad}. So we have to prove
  \begin{equation}\label{por probar}
    (k+1)\left(\sum_{h=1}^{k+1}(T^h)_{k+1} \tilde{\partial}^h_y F\right)+k\left(\sum_{h=1}^{k}(T^h)_k \tilde{\partial}^h_y F\right)
    =\tilde{\partial}^1\left(\sum_{h=1}^{k}(T^h)_k \tilde{\partial}^h_y F\right), \quad\text{for $k\ge 1$.}
  \end{equation}
  Note that $\tilde{\partial}^1=\tilde{\partial}^1_x+\tilde{\partial}^1_y$ is a derivation which commutes with $\tilde{\partial}^h_y$, and so
  $$
    \tilde{\partial}^1\left((T^h)_k \tilde{\partial}^h_y F\right)= \tilde{\partial}^1_x\left((T^h)_k\right) \tilde{\partial}^h_y F +
    (T^h)_k \tilde{\partial}^h_y\tilde{\partial}^1 F,
  $$
  since $(T^h)_k\in K[[x]]$ implies $\tilde{\partial}^1_y\left((T^h)_k\right)=0$. Similarly $f(x)^{v_0}$ is a constant in $(K[[x]])[[y]]$,
  and so, by~\eqref{partial tilde propiedad}, we have
  $$
    \tilde{\partial}^h_y\tilde{\partial}^1 F=\tilde{\partial}^h_y\left(f(x)^{v_0}\tilde{\partial}^1_y F\right)=
    f(x)^{v_0}\tilde{\partial}^h_y\left(\tilde{\partial}^1_y F\right)=
    f(x)^{v_0}\left((h+1)\tilde{\partial}^{h+1}_y F+h\tilde{\partial}^{h}_y F\right).
  $$
  It follows that
  \begin{align*}
    \tilde{\partial}^1\left(\sum_{h=1}^{k}(T^h)_k \tilde{\partial}^h_y F\right)& =
    \sum_{h=1}^{k}\left(\tilde{\partial}^1_x(T^h)\right)_k \tilde{\partial}^h_y F
    +f(x)^{v_0}\left(\sum_{h=1}^{k}(h+1)(T^h)_k \tilde{\partial}^{h+1}_y F
    +\sum_{h=1}^{k}h(T^h)_k \tilde{\partial}^h_y F\right)\\
    & =
    \sum_{h=1}^{k}\left(\tilde{\partial}^1_x(T^h)\right)_k \tilde{\partial}^h_y F
    +\sum_{h=2}^{k+1}h f(x)^{v_0}(T^{h-1})_k \tilde{\partial}^{h}_y F
    +\sum_{h=1}^{k}hf(x)^{v_0}(T^h)_k \tilde{\partial}^h_y F.
  \end{align*}
  In order to prove~\eqref{por probar}, it suffices to prove that the coefficients corresponding to $\tilde{\partial}^h_y F$ on both sides
  coincide for $1\le h \le k+1$. So we will prove
  \begin{align}
    (k+1)T_{k+1}+k T_k & = (\tilde{\partial}^1 T)_k+f(x)^{v_0}T_k, \label{caso h igual 1} \\
    (k+1)(T^h)_{k+1}+k (T^h)_k & = \left(\tilde{\partial}^1(T^h)\right)_k+h f(x)^{v_0}(T^{h-1})_k+h f(x)^{v_0}(T^{h})_k,
    \label{caso h entre 1 y k}\\
    (k+1)(T^{k+1})_{k+1} & = (k+1) f(x)^{v_0}(T^{k})_k, \label{caso h igual a k+1}
  \end{align}
  where~\eqref{caso h igual 1} corresponds to the case $h=1$, the equality~\eqref{caso h entre 1 y k} will hold for
  $1<h\le k$, and~\eqref{caso h igual a k+1} corresponds to the case $h=k+1$.
  In order to prove~\eqref{caso h igual 1}, it suffices to consider the equality~\eqref{T derivado} at degree $k\ge 1$,
  taking into account that $f(x)^{v_0}$ is a constant in $(K[[x]])[[y]]$ and so $(f(x)^{v_0})_k=0$. Now
  we take $h$ with $1<h\le k$ and compute
  \begin{align*}
    (k+1)(T^h)_{k+1}+k (T^h)_k & = ((T^h)_y)_k+((T^h)_y)_{k-1}= \left((1+y)(T^h)_y\right)_k=\left(hT^{h-1}(1+y)T_y\right)_k \\
    & = \left(hT^{h-1}(1+y)T_y\right)_k = \left(hT^{h-1}\left(\tilde{\partial}^1_x T+f(x)^{v_0}(1+T)\right)\right)_k  \\
    & = \left(\tilde{\partial}^1(T^h)\right)_k+h f(x)^{v_0}(T^{h-1})_k+h f(x)^{v_0}(T^{h})_k,
  \end{align*}
  where the fifth equality follows from~\eqref{T derivado}. This proves~\eqref{caso h entre 1 y k}. Finally, since $T_0=0$ and $T_1=f(x)^{v_0}$,
  for all $j\ge 1$ we have $(T^j)_j =f(x)^{jv_0}$, from which~\eqref{caso h igual a k+1} follows, concluding the proof.
\end{proof}

\begin{corollary} \label{prueba de la igualdad principal para F}
  The equality~\eqref{igualdad principal para las series de F} holds for all $j>0$.
\end{corollary}

\begin{proof}
  By~\eqref{T propiedad} we have $\ov{f}(F)=T(\ov{f}(y))$ and so
  $$
    \ov{f}(F)^i=T(\ov{f}(y))^i=T^i(\ov{f}(y))=\sum_{k\ge i} (T^i)_k(\ov{f}(y))^k.
  $$
  It follows that
  \begin{align*}
    \sum_{i\ge 1}\left( \ov{f}(F)^i \right)_j\tilde{\partial}^i_y F = &
    \sum_{i\ge 1}\left( \sum_{k\ge i} (T^i)_k(\ov{f}(y))^k \right)_j\tilde{\partial}^i_y F
    =  \sum_{k\ge 1} \sum_{i=1}^k (T^i)_k\left((\ov{f}(y))^k \right)_j\tilde{\partial}^i_y F \\
    = & \sum_{k\ge 1} \left((\ov{f}(y))^k \right)_j \left(\sum_{i=1}^k (T^i)_k \tilde{\partial}^i_y F\right)
    = \sum_{k\ge 1} \left((\ov{f}(y))^k \right)_j  \tilde{\partial}^k F=
    \sum_{i\ge 1} \left((\ov{f}(y))^i \right)_j  \tilde{\partial}^i F,\\
  \end{align*}
  as desired.
\end{proof}

\begin{proof}[Proof of Theorem~\ref{teorema principal}]
  By Corollary~\ref{prueba de la igualdad principal para F} the equality~\eqref{igualdad principal para las series de F} holds for all $j>0$, and
  so by Remark~\ref{comparativo} the equality~\eqref{Lo que falta para partial en G} is true for all $j>0$.
  Combined with~\eqref{Gij como F}, this implies
  $$
    \left(\partial^i G\right)_j=\left(\partial^j G\right)_i,
  $$
  for all $i,j>0$, which by Proposition~\ref{remark importante1} implies
  $$
    R(i,j,k)=(\partial^j G)_{ik}=\left((\partial^j G)_k\right)_i=\left((\partial^k G)_j\right)_i=(\partial^k G)_{ij}= R(i,k,j),\quad\text{for $i,j,k>0$.}
  $$
    By Remark~\ref{caso ijk cero} this concludes the proof of
    Theorem~\ref{teorema principal}.
\end{proof}

\section{The case $\frak{p}_{11}^1\ne 0$}
\setcounter{equation}{0}
In this section we will classify completely the case $\frak{p}_{11}^1\ne 0$.
First we will prove that in this case the solution is involutive
(i.e. that $\frak{d}=\frak{p}$). Then we will prove that the $q$-cycle coalgebra is equivalent,
via Remark~\ref{equivalencia de q brazas}, to a unique standard cycle coalgebra of degree $v_0=1$.

By Proposition~\ref{prop casos principales}
and Remark~\ref{resumen}, we know that
$\frak{p}_{j0}^j=\frak{d}_{j0}^j=1$ for all $j$; and by
Proposition~\ref{d111 igual a p111} we know that $\frak{d}_{11}^1=\frak{p}_{11}^1\ne 0$. By
Proposition~\ref{formula para d1j1} this implies that
\begin{equation}\label{formula para p1j1}
  \frak{p}_{j1}^{j}=\frak{d}_{j1}^{j}=j\frak{p}_{11}^1\quad\text{for all $j$.}
\end{equation}

\begin{proposition}\label{d1,gamma^gamma ne 0 implica muchos d se anulan}
  We have $\frak{d}_{i+d,0}^{i}=\frak{p}_{i+d,0}^{i}=0$ for all $i$ and $d>0$.
\end{proposition}

\begin{proof}
  We will proceed by induction. Assume $\frak{p}_{h,0}^{l}=0$ if $0<h-l<d$ or $h-l=d$ and $l\le i-1$.
  If $i=1$, then from the first equality in~\eqref{desarrollo para p0jk} with $j=d+1$ and $i=1$, we obtain
  $$
    \frak{p}_{d+1,0}^{1}\frak{p}_{11}^{1}+\frak{p}_{d+1,0}^{d+1}\frak{p}_{1,d+1}^{1}  =\frak{p}_{1,d+1}^{1}.
  $$
  Since $\frak{p}_{d+1,0}^{d+1}=1$ and $\frak{p}_{11}^{1}\ne 0$, this yields $\frak{p}_{d+1,0}^{1}=0$.
  Assume $i>1$. Then, from the equality~\eqref{eq: comultiplicative map} and the inductive hypothesis,
  we have
  $$
    \frak{p}_{d+i,0}^{i}=\sum_{j=1}^{d+1}\frak{p}_{j,0}^1\frak{p}_{d+i-j,0}^{i-1}=
    \frak{p}_{1,0}^1\frak{p}_{d+i-1,0}^{i-1}+\frak{p}_{d+1,0}^1\frak{p}_{i-1,0}^{i-1}=0,
  $$
  where the second equality follows from the case $h-l<d$  and the last equality, from the case $h-l=d$ and
  $l\le i-1$. The same argument shows that $\frak{d}_{j0}^k=\delta_{jk}$.
\end{proof}

From now on until the end of the section, we will use, without mentioning it, that
$$
  \frak{p}_{j0}^k=\frak{d}_{j0}^k=\delta_{jk}.
$$
\begin{proposition}\label{p1,gamma^gamma ne 0 implica p1 igual d1}
  We have $\frak{d}_{b1}^c=\frak{p}_{b1}^c$ for all $b,c$.
\end{proposition}

\begin{proof}
  For $b=c$ this follows from Propositions~\ref{d111 igual a p111} and~\ref{formula para d1j1}. Assume that it
  is true for $b-c<d$.

  Specializing~\eqref{segunda en m igual a 1} in $j=k=1$ yields
  $$
    \frak{p}_{11}^{1}\frak{d}_{i1}^{1} + \sum_{h=1}^i  \frak{p}_{i1}^{h}\frak{d}_{h1}^{1} =
    \frak{d}_{11}^{1}\frak{p}_{i1}^{1} + \sum_{h=1}^i  \frak{d}_{i1}^{h}\frak{p}_{h1}^{1}
  $$
  When $i=d+1$, by the inductive hypothesis the terms with $1<h<i$ cancel out, and so
  $$
    \frak{p}_{11}^{1}\frak{d}_{d+1,1}^{1}+\frak{p}_{d+1,1}^{1}\frak{d}_{11}^{1}+ \frak{p}_{d+1,1}^{d+1}\frak{d}_{d+1,1}^{1}=
    \frak{d}_{11}^{1}\frak{p}_{d+1,1}^{1}+\frak{d}_{d+1,1}^{1}\frak{p}_{11}^{1}+ \frak{d}_{d+1,1}^{d+1}\frak{p}_{d+1,1}^{1}.
  $$
  Therefore $\frak{p}_{d+1,1}^{d+1}\frak{d}_{d+1,1}^{1}=\frak{d}_{d+1,1}^{d+1}\frak{p}_{d+1,1}^{1}$, and so
  $$
    \frak{d}_{d+1,1}^{1}=\frak{p}_{d+1,1}^{1},
  $$
  since $\frak{p}_{d+1,1}^{d+1}=\frak{d}_{d+1,1}^{d+1}\ne 0$. It remains to check that
  $\frak{d}_{b1}^c=\frak{p}_{b1}^c$ when $c>1$ and $b-c=d$.
  But by equality~\eqref{eq: comultiplicative map}
  $$
    \frak{p}_{b1}^c=\sum_{\substack{i_1+i_2=b \\ j_1+j_2=1 \\ i_1+j_1\ge 1 \\ i_2+j_2\ge c-1}}
    \frak{p}_{i_1 j_1}^1 \frak{p}_{i_2 j_2}^{c-1}=\frak{p}_{b-1,1}^{c-1}+\frak{p}_{b-c+1,1}^{1},
  $$
  where the last equality follows using that $\frak{p}_{j0}^{k}=\delta_{jk}$. Similarly
  $$
    \frak{d}_{b1}^c=\frak{d}_{b-1,1}^{c-1}+\frak{d}_{b-c+1,1}^{1},
  $$
  and so $\frak{d}_{b1}^c=\frak{p}_{b1}^c$ by induction on $c$.
\end{proof}
\begin{proposition}\label{p1,gamma^gamma ne 0 implica pr1 igual dr1}
  We have $\frak{d}_{1r}^1=\frak{p}_{1r}^1$ for all $r$.
\end{proposition}

\begin{proof}
  We proceed by induction on $r$. For $r=1$ this is true by Proposition~\ref{d111 igual a p111}. Assume that
  $\frak{d}_{1h}^1=\frak{p}_{1h}^1$ for $h<r$.
  On one hand, specializing~\eqref{segunda en m igual a 1} in $i=1$, $j=1$ and $k=r$, yields
  \begin{equation}\label{j igual a r}
    \frak{p}_{11}^{1}\frak{d}_{1r}^{1}+\sum_{l=1}^r  \frak{p}_{r1}^{l} \frak{d}_{1l}^{1}=  \sum_{c+d=r}
    \frak{d}_{1c}^{1}\frak{d}_{1d}^{1}\frak{p}_{11}^{1}.
  \end{equation}
  On the other hand, specializing~\eqref{segunda en m igual a 1} in $i=1$, $j=r$ and $k=1$, yields
  \begin{equation}\label{i igual a r}
    \sum_{a+b=r}
    \frak{p}_{1a}^{1}\frak{p}_{1b}^{1}\frak{d}_{11}^{1}= \frak{d}_{11}^{1}\frak{p}_{1r}^{1}+\sum_{l=1}^r  \frak{d}_{r1}^{l} \frak{p}_{1l}^{1}.
  \end{equation}
  Subtracting the right hand side of~\eqref{i igual a r} from the left hand side of~\eqref{j igual a r}, and
  subtracting the left hand side of~\eqref{i igual a r} from the right hand side of~\eqref{j igual a r}, and
  using the inductive hypothesis and Proposition~\ref{p1,gamma^gamma ne 0 implica p1 igual d1}, we obtain
  $$
    \frak{p}_{11}^1\frak{d}_{1r}^1-\frak{d}_{11}^1\frak{p}_{1r}^1+
    \frak{p}_{r1}^r\frak{d}_{1r}^1-\frak{d}_{r1}^r\frak{p}_{1r}^1=    2(\frak{d}_{1r}^1\frak{p}_{11}^1-\frak{p}_{1r}^1\frak{d}_{11}^1).
  $$
  By~\eqref{formula para p1j1}, we have
  $$
    (r+1)\frak{p}_{11}^1(\frak{d}_{1r}^1-\frak{p}_{1r}^1) = 2\frak{p}_{11}^1(\frak{d}_{1r}^1-\frak{p}_{1r}^1),
  $$
  from which $\frak{d}_{1r}^1=\frak{p}_{1r}^1$ follows, since $r>1$ and $\frak{p}_{11}^1\ne 0$.
\end{proof}

\begin{theorem}\label{teorema solucion involutiva}
  If $\frak{p}_{11}^{1}\ne 0$, then the solution is involutive.
\end{theorem}

\begin{proof}
  We must prove that $\frak{d}=\frak{p}$, or, equivalently, that
  \begin{equation}\label{igualdad entre pes y des}
    \frak{d}_{ik}^j=\frak{p}_{ik}^j \quad\text{for all $i,j,k$.}
  \end{equation}
  We will prove~\eqref{igualdad entre pes y des}  by induction on $r:=i+k$. By
  Propositions~\ref{p1,gamma^gamma ne 0 implica p1 igual d1}
  and~\ref{p1,gamma^gamma ne 0 implica pr1 igual dr1} we know that~\eqref{igualdad entre pes y des}
  is true for all $r$, if $i\le 1$ or $k\le 1$ (in particular, it is true for $r\le 3$). Assume that
  $\frak{d}_{ab}^c=\frak{p}_{ab}^c$
  when $a+b<r$ and take some $i,k$ with $i+k=r$ and $i,k>1$. Specializing~\eqref{primera en m igual a 1} in
  $j=1$, we obtain
  \begin{equation}\label{primera con i igual a 1}
    \sum_{a+b=1} \sum_{h=1}^{i} \sum_{l=1}^{k}\frak{p}_{ia}^h \frak{d}_{kb}^l \frak{p}_{hl}^1=\sum_{c+d=k}
    \sum_{h=1}^{i} \frak{p}_{ic}^h \frak{p}_{1d}^1 \frak{p}_{h1}^1.
  \end{equation}
  Similarly, specializing~\eqref{tercera en m igual a 1} in $j=1$, we obtain
  \begin{equation}\label{segunda con i igual a 1}
    \sum_{a+b=1} \sum_{h=1}^{i} \sum_{l=1}^{k}\frak{d}_{ia}^h \frak{p}_{kb}^l \frak{d}_{hl}^1=\sum_{c+d=k}
    \sum_{h=1}^{i} \frak{d}_{ic}^h \frak{d}_{1d}^1 \frak{d}_{h1}^1.
  \end{equation}
  Now we subtract the equality~\eqref{primera con i igual a 1} from the
  equality~\eqref{segunda con i igual a 1}. Consider the left hand side. By the inductive hypothesis and the
  fact that $i+a<i+k=r$ and
  $k+b<k+i=r$, the only terms that survive correspond to $h+l=r$, i.e., the cases when $h=i$ and $l=k$. So the
  left hand side reads
  \begin{equation}
        \label{lado izquierdo}
    LHS = \sum_{a+b=1} \frak{p}_{ia}^i \frak{d}_{kb}^k \frak{p}_{ik}^1-\sum_{a+b=1} \frak{d}_{ia}^i \frak{p}_{kb}^k \frak{d}_{ik}^1
    =(\frak{p}_{i0}^i \frak{d}_{k1}^k+ \frak{p}_{i1}^i \frak{d}_{k0}^k)
    (\frak{p}_{ik}^1-\frak{d}_{ik}^1)=(k+i)\frak{p}_{11}^1 (\frak{p}_{ik}^1-\frak{d}_{ik}^1),
  \end{equation}
  where we use that $\frak{d}_{t1}^t=\frak{p}_{t1}^t=t\frak{p}_{11}^1$ for all $t\ge 1$.

  Now consider the right hand side. By Propositions~\ref{p1,gamma^gamma ne 0 implica p1 igual d1}
  and~\ref{p1,gamma^gamma ne 0 implica pr1 igual dr1} and the inductive hypothesis the only term that
  survives is the term corresponding to $h=1$ and $c=k$. In fact, for $c<k$ the terms cancel out, and if $h>1$, then
  by~\eqref{eq: comultiplicative map} we have
  \begin{equation}\label{descomposicion coalgebra}
    \frak{d}_{ic}^h=\sum_{\substack{i_1+i_2=i \\ j_1+j_2=c \\ i_1+j_1\ge 1 \\ i_2+j_2\ge l-1}}
    \frak{d}_{i_1 j_1}^1 \frak{d}_{i_2 j_2}^{h-1},
  \end{equation}
  and necessarily $i_1 +j_1<i+c\le r$ and $ i_2 +j_2<i+c\le r$, since for example $i_1 +j_1=i+c$ implies that
  $ i_2 +j_2=0$, which is impossible.
  So, on the right hand side of the substraction we have
  $$
    RHS = \frak{p}_{ik}^1 \frak{p}_{10}^1 \frak{p}_{11}^1-\frak{d}_{ik}^1 \frak{d}_{10}^1 \frak{d}_{11}^1
    = \frak{p}_{11}^1 (\frak{p}_{ik}^1-\frak{d}_{ik}^1).
  $$
  Combining this with~\eqref{lado izquierdo}, we obtain
  $(k+i-1)\frak{p}_{11}^1 (\frak{p}_{ik}^1-\frak{d}_{ik}^1)=0$, which implies that
  $\frak{d}_{ik}^1=\frak{p}_{ik}^1$,
  because $k+i>1$.
  If $j>1$, then the inductive hypothesis
  and~\eqref{descomposicion coalgebra} yield $\frak{d}_{ik}^j=\frak{p}_{ik}^j$ which completes the inductive
  step and finishes the proof.
\end{proof}

\begin{proposition}\label{todo depende de pi11}
  If   $\frak{p}_{11}^{1}\ne 0$, then $\frak{d}=\frak{p}$, $\frak{p}_{j0}^1=\delta_{j1}$ for all $j$, and $\frak{p}$ depends only on
  $\frak{p}_{1k}^{1}$, $k=1,\dots,n-1$,  via the recursive formulas
  \begin{align}
    & \frak{p}_{j1}^1 \frak{p}_{11}^1= \frak{p}_{11}^1 \sum_{a=0}^{j-1} \frak{p}_{1a}^1 \frak{p}_{1,j-a}^1 - \sum_{l=2}^{j} l\frak{p}_{j-l+1,1}^1
    \frak{p}_{1l}^1,\quad\text{for $j>1$}\label{formula de pj11 dependiendo de p111}\\
    &\frak{p}_{ij}^1(i+j-1)\frak{p}_{11}^1 =\sum_{\substack{a+b=j\\ 1\le h\le j \\ (a,h)\ne
        (j,1)}} \frak{p}_{ia}^h \frak{p}_{1b}^1 \frak{p}_{h1}^1
    -  \sum_{\substack{c+d=1\\ 1\le h\le i \\
        1\le l\le j\\ (h,l)\ne (i,j)}} \frak{p}_{ic}^h \frak{p}_{jd}^l \frak{p}_{hl}^1 ,
        \quad\text{ for $i,j>1$.}\label{formula de coeficientes dependiendo de p111}
    \shortintertext{and}
    &\frak{p}_{jk}^l=\sum_{\substack{j_1+j_2=j \\ k_1+k_2=k \\ j_1+k_1\ge 1 \\ j_2+k_2\ge l-1}}
    \frak{p}_{j_1 k_1}^1 \frak{p}_{j_2 k_2}^{l-1},\quad\text{ for $l>1$.}
    \label{formula para coeficientes p con l mayor que 1}
  \end{align}
\end{proposition}

\begin{proof}
    The equalities $\frak{d}=\frak{p}$ and $\frak{p}_{j0}^1=\delta_{j1}$ for all $j$, follow from Theorem~\ref{teorema solucion involutiva}
    and Proposition~\ref{d1,gamma^gamma ne 0 implica muchos d se anulan}, respectively.
    Specializing~\eqref{primera en m igual a 1} in $k=1$ for some $i,j>0$, we obtain
    \begin{equation}\label{primera formula 1}
        \sum_{a+b=j}
        \sum_{h=1}^{i} \frak{p}_{ia}^h \frak{p}_{1b}^1 \frak{p}_{h1}^1= \sum_{c+d=1} \sum_{h=0}^{i} \sum_{l=1}^{j}\frak{p}_{ic}^h \frak{p}_{jd}^l
        \frak{p}_{hl}^1.
    \end{equation}
    Taking $i=1$ we obtain
    $$
        \sum_{a+b=j}
        \frak{p}_{1a}^1 \frak{p}_{1b}^1 \frak{p}_{11}^1= \sum_{c+d=1} \sum_{l=1}^{j}\frak{p}_{1c}^1 \frak{p}_{jd}^l \frak{p}_{1l}^1
        = \frak{p}_{11}^1 \frak{p}_{j0}^j \frak{p}_{1j}^1+ \sum_{l=1}^{j} \frak{p}_{j1}^l \frak{p}_{1l}^1,
    $$
    and so
    $$
      \frak{p}_{11}^1 \sum_{a=0}^{j-1} \frak{p}_{1a}^1 \frak{p}_{1,j-a}^1 =  \frak{p}_{j1}^1 \frak{p}_{11}^1 +
      \sum_{l=2}^{j} \frak{p}_{j1}^l \frak{p}_{1l}^1,
    $$
    from which~\eqref{formula de pj11 dependiendo de p111} follows, using that one can check using~\eqref{eq: comultiplicative map} that
    $\frak{p}_{j1}^l=l\frak{p}_{j-l+1,1}^1$.

    Since the left hand side of~\eqref{primera formula 1} is
    $$
        \frak{p}_{ij}^1 \frak{p}_{10}^1 \frak{p}_{11}^1+\sum_{\substack{a+b=j\\ 1\le h\le j \\ (a,h)\ne
        (j,1)}} \frak{p}_{ia}^h \frak{p}_{1b}^1 \frak{p}_{h1}^1,
    $$
    and the right hand side of~\eqref{primera formula 1} is
    $$
        \sum_{c+d=1} \frak{p}_{ic}^i \frak{p}_{jd}^j \frak{p}_{ij}^1+\sum_{\substack{c+d=1\\ 1\le h\le i \\
        1\le l\le j\\ (h,l)\ne (i,j)}} \frak{p}_{ic}^h \frak{p}_{jd}^l \frak{p}_{hl}^1,
    $$
    equality~\eqref{formula de coeficientes dependiendo de p111} follows from the fact that, since
    $\frak{p}_{t1}^t=t\frak{p}_{11}^1$ and $\frak{p}_{t0}^t=1$ for all $t\ge 1$, we have
    $$
        \sum_{c+d=1}  \frak{p}_{ic}^i \frak{p}_{jd}^j \frak{p}_{ij}^1 - \frak{p}_{ij}^1
        \frak{p}_{11}^1=(i+j-1)\frak{p}_{11}^1 \frak{p}_{ij}^1.
    $$
    Finally, equality~\eqref{formula para coeficientes p con l mayor que 1}
    follows immediately from~\eqref{eq: comultiplicative map}.
\end{proof}

\begin{corollary} \label{cor i0 igual a 1}
  If   $\frak{p}_{11}^{1}\ne 0$, then the given $q$-cycle coalgebra $(C,\frak{p},\frak{d})$ is equivalent, via
  Remark~\ref{equivalencia de q brazas}, to a unique standard cycle coalgebra of degree $v_0=1$.
\end{corollary}

\begin{proof}
  Taking  $f_{\lambda}$ with $\lambda=\frak{p}_{11}^{1}$  in Remark~\ref{equivalencia de q brazas}, the given $q$-cycle coalgebra
  $(C,\frak{p},\frak{d})$ is equivalent to a unique $q$-cycle coalgebra $(C,\tilde{\frak{p}},\tilde{\frak{d}})$ with $\tilde{\frak{p}}_{11}^{1}=1$. By
  Proposition~\ref{todo depende de pi11}, if we set
  $$
    f=1+x+\sum_{i=2}^{n-1} p_i x^i,
  $$
  with $p_i=\tilde{\frak{p}}_{1i}^1$, then
  the standard cycle coalgebra $\scc(f)$ coincides with $(C,\tilde{\frak{p}},\tilde{\frak{d}})$,
  which concludes the proof.
\end{proof}

\begin{example} \label{The case in which all the parameters are equal to 1}
  The standard cycle coalgebra of degree $v_0=1$ with $\frak{p}_{1j}^1=1$ for all $j<n$. We assert that the coefficients are given by
  \begin{equation}\label{solution condition 1}
    \frak{p}_{ij}^k=\begin{cases}
                  \binom{j+k-1}{k-1} & \mbox{if } i=k>0, \\
                  1 & \mbox{if } i=k=j=0,\\
                   0 & \mbox{if } i=k=0 \text{ and } j>0 \\
                   0& \mbox{if } i\ne k.\\
                \end{cases}
  \end{equation}
  By Proposition~\ref{todo depende de pi11}, in order to check this we only must prove
  that~\eqref{solution condition 1} satisfies
  \begin{equation}\label{primera involutivo}
    \sum_{a+b=j} \sum_{h=0}^{i} \sum_{l=0}^{k} \frak{p}_{ia}^h \frak{p}_{kb}^l \frak{p}_{hl}^1 =
    \sum_{c+d=k} \sum_{h=0}^{i} \sum_{l=0}^{j} \frak{p}_{ic}^h \frak{p}_{jd}^l \frak{p}_{hl}^1
  \end{equation}
  for all $i,j,k$, since then the above mentioned proposition says that this is the unique solution with
  $\frak{p}_{1j}^1=1$ for all $j>1$. But this follows by a direct computation using that
  $$
    \sum_{b=0}^j\binom{b+k-1}{k-1}=\sum_{d=0}^k\binom{d+j-1}{j-1}=\binom{k+j}{j}.
  $$
\end{example}

\begin{example} \label{The case of vanishing parameters}
   The standard cycle coalgebra of degree $v_0=1$ with $\frak{p}_{1j}^1=0$ for all $j>1$. In this case
  \begin{equation}\label{solution condition}
    \frak{p}_{ij}^k=\frak{d}_{ij}^k=\binom{i-1}{k-1}\binom{k}{i-j}\quad\text{if $i,j,k\ge 1$.}
  \end{equation}
  One can prove directly by a very lengthy proof, that~\eqref{solution condition}
  satisfies~\eqref{primera involutivo}, which shows that the formulas for the coefficients are correct.
  Alternatively consider the equality~\eqref{relation f and gi}, which in this case reads $g(x)f'(x)=f(x) \ov{f}(x)$. Since
  $f(x)=1+x$, we obtain
  $$
  g(x)=g(x)f'(x)=f(x)\ov{f}(x)=(x+1)x.
  $$
  A direct inductive argument shows that
  $$
    g_k(x)\coloneqq x^{k-1}g(x)=x^{k+1}+x^k
  $$
  for $k\ge 1$, satisfy the inductive definition~\eqref{definition of G por grado} of $g_k$, and so
  $$
    \frak{p}_{ij}^1=G_{ij}=(g_j(x))_i=\binom{1}{i-j},
  $$
  for $i+j>0$. We also have $\frak{p}_{ij}^k=\binom{i-1}{k-1} \binom{k}{i-j}$. In fact, assume that the formula is valid for $k-1$ and some $k\ge 2$.
  Then, by~\eqref{inductivo coproducto1} we have
  $$
    \frak{p}_{ij}^k =\sum_{i_1=1}^{i-1}\sum_{j_1+j_2=j} \frak{p}_{i_1,j_1}^1 \frak{p}_{i-i_1,j_2}^{k-1}.
  $$
  But
  $$
    \frak{p}_{i_1,j_1}^1=
      \begin{cases}
         1, & \mbox{if } j_1=i_1\text{ or }j_1=i_1-1 \\
         0, & \mbox{otherwise},
      \end{cases}
  $$
  and so,
  \begin{eqnarray*}
    \frak{p}_{ij}^k
    &=&
    \sum_{i_1=1}^{i-1} \frak{p}_{i-i_1,j-i_1}^{k-1}
    + \sum_{i_1=1}^{i-1} \frak{p}_{i-i_1,j-i_1+1}^{k-1}\\
    &=&
    \sum_{i_1=1}^{i-1} \binom{i-i_1-1}{k-2} \binom{k-1}{i-j}
    + \sum_{i_1=1}^{i-1} \binom{i-i_1-1}{k-2} \binom{k-1}{i-j-1} \\
    &=&
    \left(\binom{k-1}{i-j}+\binom{k-1}{i-j-1}\right)\sum_{l=0}^{i-2} \binom{l}{k-2}\\
    &=&
    \binom{k}{i-j} \binom{i-1}{k-1},
  \end{eqnarray*}
  as desired.
\end{example}

\begin{remark}
  If $v_0=1$, then in $\scc(f)$ the coefficients $\{\frak{p}_{1k}^1\}_{k\ge 2}$ determine the coefficients
  $\{\frak{p}_{k1}^1\}_{k\ge 2}$, but it is also true that the coefficients $\{\frak{p}_{k1}^1\}_{k\ge 2}$ determine the
  coefficients $\{\frak{p}_{1k}^1\}_{k\ge 2}$, and hence all of $(C,\frak{p},\frak{d})$.
  In fact, write $f_j=\frak{p}_{1k}^{1}$. Since $f_0=1$, we have
  $$
  (f\ov{f})_j =-f_j+ \sum_{i=0}^{j} f_i f_{j-i} =f_j+ \sum_{i=1}^{j-1} f_i f_{j-i}\quad\text{and}\quad
  (g(x)f'(x))_j=jf_j+\sum_{i=1}^{j-1} (g(x))_{j-i+1}i f_i.
  $$
  Hence
  $$
  (j-1)f_j=\sum_{i=1}^{j-1}f_i f_{j-i}-\sum_{h=1}^{j-1} h (g(x))_{j-h+1} f_h.
  $$
  This gives an inductive formula for $f_j=\frak{p}_{1j}^{1}$ depending on the coefficients $(g(x))_i=\frak{p}_{i1}^1$, and shows
  that the coefficients $\{\frak{p}_{k1}^1\}_{k\ge 2}$ determine the coefficients
  $\{\frak{p}_{1k}^1\}_{k\ge 2}$, as desired.
\end{remark}

\section{The involutive case with $\frak{p}_{10}^1= 1$, $\frak{p}_{1i}^1=0$ for $0<i<v_0$ and
$\frak{p}_{1,v_0}^1\ne 0$.}
\setcounter{equation}{0}

In this section we assume $\frak{p}=\frak{d}$. The braid equations~\eqref{primera en m igual a 1}, \eqref{segunda en m igual a 1}
and~\eqref{tercera en m igual a 1}, reduce to
  \begin{equation}\label{braid equation involutivo}
    \sum_{a+b=j} \sum_{h=0}^{i} \sum_{l=0}^{k} \frak{p}_{ia}^h \frak{p}_{kb}^l \frak{p}_{hl}^1 =
    \sum_{c+d=k} \sum_{h=0}^{i} \sum_{l=0}^{j} \frak{p}_{ic}^h \frak{p}_{jd}^l \frak{p}_{hl}^1.
  \end{equation}
We also fix $v_0$ with $1<v_0<n$ and assume that
\begin{itemize}
  \item[a)]  $\frak{p}_{10}^1= 1$,
  \item[b)] $\frak{p}_{1i}^1=0$ for $0<i<v_0$,
  \item[c)] $\frak{p}_{1,v_0}^1\ne 0$.
\end{itemize}
We will prove that the given $q$-cycle coalgebra $(C,\frak{p},\frak{d})$ is equivalent, via
  Remark~\ref{equivalencia de q brazas}, to a uniquely determined standard cycle coalgebra of degree $v_0$.
\begin{lemma}
  We have
  \begin{equation}\label{pikk}
    \frak{p}_{ki}^k=0, \quad\text{if\ $0<i<v_0$}
  \end{equation}
  and
  \begin{equation}\label{pi0kk}
    \frak{p}_{k,v_0}^k=k \frak{p}_{1,v_0}^1.
  \end{equation}
\end{lemma}

\begin{proof}
  By equality~\eqref{eq: comultiplicative map} and Proposition~\ref{lado derecho} we have for $k\ge 1$
  \begin{equation}\label{pikk comultiplicativo}
    \frak{p}_{ki}^k = \sum_{i_1+i_2=i} \frak{p}_{1,i_1}^1 \frak{p}_{k-1,i_2}^{k-1}.
  \end{equation}
  If $0<i<v_0$ then this gives
  $\frak{p}_{ki}^k =\frak{p}_{10}^1 \frak{p}_{k-1,i}^{k-1}= \frak{p}_{k-1,i}^{k-1}$ for all $k>1$, and so,
  since $\frak{p}_{1i}^1=0$, we obtain~\eqref{pikk}.
  For  $i=v_0$ the equality~\eqref{pikk comultiplicativo} yields
  $$
    \frak{p}_{k,v_0}^k =
    \frak{p}_{10}^1 \frak{p}_{k-1,v_0}^{k-1}+ \frak{p}_{1,v_0}^{1} \frak{p}_{k-1,0}^{k-1}
    \quad \text{ for all $k>1$,}
  $$
  and a direct inductive argument proves~\eqref{pi0kk}.
\end{proof}

\begin{proposition} \label{prop p0 es delta}
  We have
  \begin{equation} \label{p0 es delta}
    \frak{p}_{j0}^k=\delta_{kj}.
  \end{equation}
\end{proposition}

\begin{proof}
  If $j\le k$, then~\eqref{p0 es delta} holds by Remark~\ref{resumen}.
  So we have to prove that $\frak{p}_{k+d,0}^k=0$ for all $d>0$. Assume that
  $$
    \frak{p}_{h0}^l=0,\quad\text{if $0<h-l<d$.}
  $$
  Then~\eqref{eq: comultiplicative map} yields
  $$
  \frak{p}_{k+d,0}^k= \frak{p}_{10}^1 \frak{p}_{k-1+d,0}^{k-1}+ \frak{p}_{1+d,0}^1 \frak{p}_{k-1,0}^{k-1},
  $$
  and a direct inductive arguments gives
  \begin{equation}\label{p0kd}
    \frak{p}_{k+d,0}^k= k \frak{p}_{1+d,0}^1.
  \end{equation}
  From~\eqref{braid equation involutivo} with $i=1$ and $k=0$, we obtain
  $$
    \frak{p}_{1,j}^1 =  \sum_{l=0}^{j} \frak{p}_{10}^1 \frak{p}_{j0}^l \frak{p}_{1l}^1
  $$
  and for $j=v_0+d$ this yields
  $$
  \frak{p}_{1,v_0+d}^1 =  \sum_{l=v_0}^{v_0+d} \frak{p}_{j0}^l \frak{p}_{1l}^1.
  $$
  But by assumption $\frak{p}_{v_0+d,0}^h=0$ for $v_0<h<v_0+d$, hence
  $$
    \frak{p}_{1,v_0+d}^1 =  \frak{p}_{v_0+d,0}^{v_0} \frak{p}_{1,v_0}^1 +
    \frak{p}_{v_0+d,0}^{v_0+d} \frak{p}_{1,v_0+d}^1,
  $$
  which gives $\frak{p}_{v_0+d,0}^{v_0}=0$, since $\frak{p}_{1,v_0}^1\ne 0$. Finally,
  from~\eqref{p0kd} we obtain
  $$
    \frak{p}_{k+d,0}^{k}=\frac{k}{v_0} \frak{p}_{v_0+d,0}^{v_0}=0,
  $$
  which completes the inductive step and concludes the proof.
\end{proof}

\begin{proposition} \label{prop pijk se anula para chicos}
  We have
  \begin{equation}\label{p para i menor que i0 se anula}
    \frak{p}_{ij}^k=0,\quad \text{for $0<j<v_0$ and all $i,k$.}
  \end{equation}
\end{proposition}

\begin{proof}
  By Proposition~\ref{lado derecho} and~\eqref{pikk} the equality~\eqref{p para i menor que i0 se anula}
  holds for $i\le k$. So we have to prove that $\frak{p}_{k+d,j}^k=0$ for all $d>0$ and $0<j<v_0$, and we will do it by
  induction. So fix $d>0$ adn $j$ with $0<j<v_0$, and assume that
  $$
    \frak{p}_{k+d_1,j_1}^k=0,\quad \text{for $0<j_1<v_0$, $0\le d_1<d$  and all $k$,}
  $$
  and that
  $$
    \frak{p}_{k+d,j_1}^k=0,\quad \text{for $0<j_1<j$ and all $k$.}
  $$
  We claim that
  \begin{equation}\label{pik derivacion}
    \frak{p}_{k+d,j}^k=k \frak{p}_{1+d,j}^1.
  \end{equation}
  In fact, by~\eqref{eq: comultiplicative map} and Proposition~\ref{lado derecho}, we have
  $$
    \frak{p}_{k+d,j}^k=\sum_{\substack{d_1+d_2=d\\ j_1+j_2=j}} \frak{p}_{1+d_1,j_1}^1
    \frak{p}_{k-1+d_2,j_2}^{k-1}.
  $$
  But for $0<j_1<j$ we have $\frak{p}_{1+d_1,j_1}^1 =0$ by assumption. Moreover, by~\eqref{p0 es delta}
  for $j_1=0$ necessarily $d_1=0$ and for $j_2=0$ necessarily $d_2=0$, hence
  $$
    \frak{p}_{k+d,j}^k= \frak{p}_{10}^1  \frak{p}_{k-1+d,j}^{k-1} +
    \frak{p}_{1+d,j}^1  \frak{p}_{k-1,0}^{k-1},
  $$
  and a direct inductive argument proves~\eqref{pik derivacion}.
  On the other hand from the equality~\eqref{braid equation involutivo} with $i=1$ and $0<j<v_0$
  we obtain
  $$
  \sum_{l=v_0}^{k} \frak{p}_{kj}^l \frak{p}_{1l}^1=0.
  $$
  Taking $k=v_0+d$ and using that by assumption $\frak{p}_{v_0+d,j}^{l}=0$ for $l>v_0$ (since then
  $v_0+d-l<d$), we arrive at $\frak{p}_{v_0+d,j}^{v_0}=0$. Now by~\eqref{pik derivacion} we have
  $$
    \frak{p}_{k+d,j}^{k}= \frac{k}{v_0} \frak{p}_{v_0+d,j}^{v_0}=0,
  $$
  which finishes the inductive step and concludes the proof.
\end{proof}

\begin{proposition} \label{pi0 derivation}
  We have $\frak{p}_{i,v_0}^k=k \frak{p}_{i-k+1,v_0}^1$.
\end{proposition}

\begin{proof}
  For $i<k$ both sides vanish, for $i=k$ the equality holds by~\eqref{pi0kk}, and by~\eqref{eq: comultiplicative map} and
  Proposition~\ref{lado derecho}, for $d>0$ we have
  $$
    \frak{p}_{k+d,v_0}^k=\sum_{\substack{d_1+d_2=d\\ j_1+j_2=v_0}} \frak{p}_{1+d_1,j_1}^1
    \frak{p}_{k-1+d_2,j_2}^{k-1}.
  $$
  But for $0<j_1<v_0$ we have $\frak{p}_{1+d_1,j_1}^1 =0$ by Proposition~\ref{prop pijk se anula para chicos}. Moreover, by~\eqref{p0 es delta}
  for $j_1=0$ necessarily $d_1=0$ and for $j_2=0$ necessarily $d_2=0$, hence
  $$
    \frak{p}_{k+d,v_0}^k= \frak{p}_{10}^1  \frak{p}_{k-1+d,v_0}^{k-1} +
    \frak{p}_{1+d,v_0}^1  \frak{p}_{k-1,0}^{k-1},
  $$
  and a direct inductive argument concludes the proof.
\end{proof}

From now on we will assume that $\frak{p}^1_{1,v_0}=1$. If $K$ is algebraically closed, then by
Remark~\ref{equivalencia de q brazas} any $q$-cycle coalgebra with $\frak{p}^1_{1,v_0}\ne 0$ is equivalent to
a $q$-cycle coalgebra with $\ov{\frak{p}}^1_{1,v_0}=1$ via $f_\lambda$ with
$\lambda=\sqrt[v_0]{\frak{p}^1_{1,v_0}}$.

\begin{proposition} \label{prop relation f and gi}
  Set $g(x)= \displaystyle\sum_{i\ge 1} \frak{p}_{i,v_0}^1 x^i$ and $f(x)=\displaystyle\sum_{j\ge 0} \frak{p}_{1,j}^1 x^j$. Then
  \begin{equation}\label{relation f and gi1}
    gf'=f(f^{v_0}-1),
  \end{equation}
  where $f'$ denotes the formal derivative of $f$ in $K[[x]]$.
\end{proposition}

\begin{proof}
  Consider the equality~\eqref{braid equation involutivo} for $i=1$, $k=v_0$ and some $j>0$:
  $$
    \sum_{a+b=j} \frak{p}_{1a}^{1} \frak{p}_{v_0,b}^{v_0} \frak{p}_{1,v_0}^1=
    \sum_{l=1}^{j} \frak{p}_{10}^{1} \frak{p}_{j,v_0}^{l} \frak{p}_{1l}^{1} +
    \sum_{l=1}^{j} \frak{p}_{1,v_0}^{1} \frak{p}_{j0}^{l} \frak{p}_{1l}^{1}.
  $$
  But the left hand side yields $(f^{v_0+1})_j$, moreover
  $\sum_{l=1}^{j} \frak{p}_{1,v_0}^{1} \frak{p}_{j0}^{l} \frak{p}_{1l}^{1}=\frak{p}_{1j}^1=(f)_j$ and
  by Proposition~\ref{pi0 derivation}
  $$
    \sum_{l=1}^{j} \frak{p}_{10}^{1} \frak{p}_{j,v_0}^{l} \frak{p}_{1l}^{1}=
    \sum_{l=v_0}^{j} l \frak{p}_{j-l+1,v_0}^{1} \frak{p}_{1l}^{1}=
    \sum_{l=v_0}^{j} (g)_{j-l+1}(f')_{l-1}=(gf')_{j},
  $$
  since $(g)_0=g(0)=0$. Hence $(f^{v_0+1})_j=(gf'+f)_j$ for all $j\ge 1$. This yields
  $$
    (gf')_j=(f(f^{v_0}-1))_j
  $$
  for all $j\ge 1$. But trivially $(gf')_0=(f(f^{v_0}-1))_0$, which shows that~\eqref{relation f and gi1}
  is satisfied.
\end{proof}

\begin{theorem} \label{teorema parametros}
  If $\frak{d}=\frak{p}$, $\frak{p}_{10}^1= 1$, $\frak{p}_{1i}^1=0$ for $0<i<v_0$ and $\frak{p}_{1,v_0}^1=1$,
   then $\frak{p}$ depends only on the coefficients $\{ \frak{p}_{1i}^1\}_{i > v_0}$ in the following manner:
  \begin{itemize}
    \item[a)] $\frak{p}_{j0}^k=\delta_{jk}$, for all $j,k$.
    \item[b)] $\frak{p}_{ij}^k=0$, for $0<j<v_0$ and all $i,k$.
    \item [c)]  Setting $f(x)=\sum_{j\ge 0} \frak{p}_{1j}^1 x^j$ the coefficients $\frak{p}_{i,v_0}^1$ are given by
        $$
            g(x)=\sum_{i\ge 1} \frak{p}_{i,v_0}^1 x^i,\quad\text{where}\quad
            g(x)=\frac{f(x)(f(x)^{v_0}-1)}{f'(x)}\in K[[x]].
        $$
    \item[d)] For $j>v_0$ and $i>1$, we have the recursive formulas
         \begin{equation} \label{formula de coeficientes dependiendo de pi011}
           (i+j-1)\frak{p}_{ij}^1 =
           \sum_{\substack{a+b=j\\ 1\le h\le i\\ (h,a)\ne (1,j)}}\frak{p}_{ia}^h \frak{p}_{v_0,b}^{v_0} \frak{p}_{h,v_0}^1-
            \sum_{\substack{c+d=v_0\\ 1\le h\le i \\ 1\le l \le j \\ (h,l)\ne (i,j)}}
            \frak{p}_{ic}^{h} \frak{p}_{jd}^l \frak{p}_{hl}^1,
         \end{equation}
    \item[e)] For $k>1$ we have the recursive formulas
        \begin{equation} \label{formula para coeficientes p con k mayor que 1}
            \frak{p}_{ij}^k=\sum_{\substack{i_1+i_2=i \\ j_1+j_2=j \\ i_1+j_1\ge 1 \\ i_2+j_2\ge k-1}}
            \frak{p}_{i_1 j_1}^1 \frak{p}_{i_2 j_2}^{k-1}.
        \end{equation}
  \end{itemize}
\end{theorem}

\begin{proof}
    The first three cases are covered by Propositions~\ref{prop p0 es delta},
    \ref{prop pijk se anula para chicos} and~\ref{prop relation f and gi}.
    Specializing~\eqref{braid equation involutivo} at $k=v_0$ we obtain
    $$
    \sum_{\substack{a+b=j\\ 1\le h\le i}}\frak{p}_{ia}^h \frak{p}_{v_0,b}^{v_0}
    \frak{p}_{h,v_0}^1=
    \sum_{\substack{c+d=v_0\\ 1\le h\le i \\ 1\le l \le j }}
    \frak{p}_{ic}^{h} \frak{p}_{jd}^l \frak{p}_{hl}^1,
    $$
    from which~\eqref{formula de coeficientes dependiendo de pi011} follows, since
    $$
    \frak{p}_{i,v_0}^i \frak{p}_{j0}^j \frak{p}_{ij}^1+
    \frak{p}_{i0}^i \frak{p}_{j,v_0}^j \frak{p}_{ij}^1=(i+j)\frak{p}_{ij}^1,
    $$
    by Proposition~\ref{pi0 derivation}.
    Finally, the equality~\eqref{formula para coeficientes p con k mayor que 1}
    follows immediately from~\eqref{eq: comultiplicative map}.
\end{proof}

\begin{corollary} \label{cor i0 mayor que 1}
  Assume that $K$ is algebraically closed.
  If $\frak{d}=\frak{p}$, $\frak{p}_{10}^1= 1$, $\frak{p}_{1i}^1=0$ for $0<i<v_0$ and $\frak{p}_{1,v_0}^1\ne 0$,
  then the given $q$-cycle coalgebra $(C,\frak{p},\frak{d})$ is equivalent, via
  Remark~\ref{equivalencia de q brazas}, to a uniquely determined standard cycle coalgebra of degree $v_0$.
\end{corollary}

\begin{proof}
  Taking  $f_{\lambda}$ with $\lambda=\sqrt[v_0]{\frak{p}^1_{1,v_0}}$ in Remark~\ref{equivalencia de q brazas}, the given $q$-cycle coalgebra
  $(C,\frak{p},\frak{d})$ is equivalent to a unique $q$-cycle coalgebra $(C,\tilde{\frak{p}},\tilde{\frak{d}})$ with
  $\frak{d}=\frak{p}$, $\frak{p}_{10}^1= 1$, $\frak{p}_{1i}^1=0$ for $0<i<v_0$ and
  $\tilde{\frak{p}}_{1,v_0}^{1}=1$. By
  Theorem~\ref{teorema parametros}, if we set
  $$
    f=1+x^{v_0}+\sum_{i=2}^{n-1} p_i x^i
  $$
  with $p_i=\tilde{\frak{p}}_{1i}^1$,
  the standard cycle coalgebra $\scc(f)$
  coincides with $(C,\tilde{\frak{p}},\tilde{\frak{d}})$,
  which concludes the proof.
\end{proof}

\section{The case $\frak{p}_{10}^1\ne 1$}
\setcounter{equation}{0}

In this section we consider the case $\frak{p}_{10}^1\ne 1$, similar results are obtained when $\frak{d}_{10}^1\ne 1$.
\begin{theorem}\label{teorema no raiz de la unidad}
  If $(\frak{p}_{10}^1)^r\ne 1$ for all $r<s$, then $\frak{p}_{ij}^k=\frak{d}_{ij}^k= 0$ whenever $i+j\le s$
  and $i,j \ne 0$. Moreover,
  the formula
  \begin{equation}\label{formula simetrica para p0r1}
    \frak{d}_{i0}^1(\frak{p}_{10}^1-\frak{p}_{i0}^i)=\sum_{h=1}^{i-1} \frak{p}_{i0}^h \frak{d}_{h0}^1 -
    \sum_{h=2}^{i} \frak{d}_{i0}^h \frak{p}_{h0}^1
  \end{equation}
  defines recursively $\frak{d}_{i0}^1$ for $1<i\le s$.
\end{theorem}

\begin{proof}
  We prove the first assertion by induction on $r\coloneqq i+j$. Assume $i,j>0$, $r=i+j\le s$ and that $\frak{p}_{ab}^k= 0$ if $a+b<r$ and $a,b>0$.
  The equality~\eqref{primera en m igual a 1} with $k=0$ and $i,j>0$ yields
  $$
     \sum_{h=1}^{i}
    \frak{p}_{ij}^h \frak{p}_{h0}^1 = \sum_{h=1}^{i} \sum_{l=1}^{j}  \frak{p}_{i0}^h \frak{p}_{j0}^l \frak{p}_{hl}^1.
  $$
  By the inductive hypothesis, all terms but the term corresponding to $h=i$ and $l=j$ vanish at the right
  hand side. By the inductive hypothesis and~\eqref{eq: comultiplicative map},
  all terms but the term corresponding to $h=1$ vanish at the left hand side, and so we have
  $$
    \frak{p}_{ij}^1 \frak{p}_{10}^1 = \frak{p}_{i0}^i \frak{p}_{j0}^j \frak{p}_{ij}^1.
  $$
  This implies
  $$
    \frak{p}_{ij}^1 \frak{p}_{10}^1(1- (\frak{p}_{10}^1)^{i+j-1})=0,
  $$
  and so $\frak{p}_{ij}^1=0$, since $i+j-1<s$.

  In order to prove $\frak{d}_{ik}^1=0$, for $i+k\le s$ and $i,k>0$, we set $r=i+k\le s$ and assume that $\frak{d}_{ab}^l= 0$ if $a+b<r$ and $a,b>0$. The
  equality~\eqref{segunda en m igual a 1} with $j=0$ and
  $i,k>0$ yields
  $$
    \sum_{h=1}^{i} \sum_{l=1}^{k}  \frak{p}_{i0}^h \frak{p}_{k0}^l \frak{d}_{hl}^1=
    \sum_{h=1}^{i} \frak{d}_{ik}^h \frak{p}_{h0}^1.
  $$
  By the inductive hypothesis, all terms but the term corresponding to $h=i$ and $l=k$ vanish at the left
  hand side. By the inductive hypothesis and~\eqref{eq: comultiplicative map},
  all terms but the term corresponding to $h=1$ vanish at the right hand side, and so we have
  $$
    \frak{p}_{i0}^i \frak{p}_{k0}^k \frak{d}_{ik}^1=\frak{d}_{ik}^1 \frak{p}_{10}^1.
  $$
  This implies
  $$
    \frak{d}_{ik}^1 \frak{p}_{10}^1((\frak{p}_{10}^1)^{k+i-1}-1)=0,
  $$
  and so $\frak{d}_{ik}^1=0$, as desired.

  Finally, the equality~\eqref{segunda en m igual a 1} with $j=k=0$, gives
  $$
    \sum_{h=1}^{i}  \frak{p}_{i0}^h \frak{d}_{h0}^1= \sum_{h=1}^{i}  \frak{d}_{i0}^h \frak{p}_{h0}^1,
  $$
  from which~\eqref{formula simetrica para p0r1} follows directly.
\end{proof}

\begin{corollary}\label{ejemplo completo no raiz de la unidad}
  Let $(\lambda_i)_{i>0}$ with $\lambda_1\ne 0$ be a family of elements of $K$ and $\mu\in K^{\times}$. If $\lambda_1^k\ne 1$ for all $0<k<n$,
  then there exists a unique $q$-cycle coalgebra $\mathcal{C}=(C,\cdot, :)$ with $\frak{p}_{i0}^1=\lambda_i$
  for all $i$ and $\frak{d}_{10}^{1}=\mu$. The other coefficients are given by
  \begin{align}
   & \frak{p}_{ij}^k=\frak{d}_{ij}^k=0 \quad\text{if $j>0$,}\label{eq se anulan coeficientes}\\
   &\frak{p}_{i0}^k=\sum_{\substack{i_1+\dots+i_k=i \\ i_1,\dots,i_k>0 }} \prod_{s=1}^k \frak{p}_{i_s,0}^1,\label{p en grado mayor}\\
   &\frak{d}_{i0}^1(\frak{p}_{10}^1-\frak{p}_{i0}^i)=\sum_{h=1}^{i-1} \frak{p}_{i0}^h \frak{d}_{h0}^1 -
    \sum_{h=2}^{i} \frak{d}_{i0}^h \frak{p}_{h0}^1 \label{igualdad eliminada}\\
   & \frak{d}_{i0}^k=\sum_{\substack{i_1+\dots+i_k=i \\ i_1,\dots,i_k>0 }} \prod_{s=1}^k \frak{d}_{i_s,0}^1.\label{d en grado mayor}
  \end{align}
  Moreover, $\frak{d}=\frak{p}$ if and only if $\frak{d}_{10}^1=\frak{p}_{10}^1$ (i.e. if $\lambda_1=\mu$).
\end{corollary}

\begin{proof}
  By Theorem~\ref{teorema no raiz de la unidad} and formula~\eqref{producto}, if $\mathcal{C}$ exists, it is
  unique and the formulas in the statement are true. From~\eqref{igualdad eliminada} it follows inductively that
  if $\frak{d}_{10}^1=\frak{p}_{10}^1$, then $\frak{d}_{l0}^1=\frak{p}_{l0}^1$ for all $l$. Then~\eqref{d en grado mayor}
  and~\eqref{p en grado mayor} imply $\frak{d}=\frak{p}$.

  Hence we are reduced to prove that
  these formulas define a  $q$-cycle coalgebra. By Lemma~\ref{reduccion a m igual a 1} we have to prove
  that~\eqref{primera en m igual a 1}, \eqref{segunda en m igual a 1}
  and~\eqref{tercera en m igual a 1} hold. Since $\frak{p}_{ij}^k=\frak{d}_{ij}^k=0$ if $j>0$, the terms
  $\mathfrak{p}_{ia}^h\mathfrak{d}_{kb}^l\mathfrak{p}_{hl}^1$ at the left hand side of~\eqref{primera en m igual a 1}
  can only be non zero when $a=b=l=0$, but then $k$ must also be zero, and so
  $\mathfrak{p}_{ia}^h\mathfrak{d}_{kb}^l\mathfrak{p}_{hl}^1$ can only be nonzero when $j=k=0$. The same argument shows
  that the all the terms on both sides of~\eqref{primera en m igual a 1}, \eqref{segunda en m igual a 1}
  and~\eqref{tercera en m igual a 1} can only be non zero when $j=k=0$. But for $j=k=0$ the
  equalities~\eqref{primera en m igual a 1} and~\eqref{tercera en m igual a 1} are trivial,
  and~\eqref{segunda en m igual a 1}  at  $j=k=0$ gives
  $$
    \sum_{h=1}^{i}  \frak{p}_{i0}^h \frak{d}_{h0}^1= \sum_{h=1}^{i}  \frak{d}_{i0}^h \frak{p}_{h0}^1,
  $$
  This last equality is equivalent to~\eqref{igualdad eliminada}, which concludes the proof.
\end{proof}

\begin{remark}
  Note that~\eqref{igualdad eliminada} and~\eqref{d en grado mayor} determine recursively the coefficients
  $\frak{d}_{r0}^1$ and $\frak{d}_{i0}^k$.
  We can combine them with~\eqref{p en grado mayor} in order to obtain
  $$
    \frak{d}_{r0}^1(\frak{p}_{10}^1-\frak{p}_{r0}^r)=\sum_{l=1}^{r-1} \left(\sum_{\substack{r_1+\dots+r_l=r\\
     r_1,\dots,r_l>0 }} \prod_{s=1}^l \frak{p}_{r_s,0}^1 \right)
     \frak{d}_{l0}^1 - \sum_{l=2}^{r}  \left(\sum_{\substack{r_1+\dots+r_l=r \\ r_1,\dots,r_l>0 }}
     \prod_{s=1}^l \frak{d}_{r_s,0}^1 \right) \frak{p}_{l0}^1,
  $$
  which shows that $\frak{d}_{r0}^1$ depends only on the parameters and on the coefficients $\frak{d}_{l0}^1$
  with $l<r$.
\end{remark}

\begin{remark}\label{simetrico}
   Since the set of braid equations~\eqref{primera en m igual a 1}, \eqref{segunda en m igual a 1}
   and~\eqref{tercera en m igual a 1} is symmetric in $\frak{p}$ and $\frak{d}$,
   when $\frak{d}_{10}^1\ne 1$, we obtain the same formulas as in Theorem~\ref{teorema no raiz de la unidad}
   and Corollary~\ref{ejemplo completo no raiz de la unidad}, interchanging $\frak{p}$ and $\frak{d}$.
\end{remark}

\section{Some examples and the case $n=3$}

The results in the previous sections provide families that classify all possible $q$-cycles coalgebras under some given conditions.
In this section we consider the cases in the classification table which are not completely classified by the previous results,
and provide examples for these cases.
All examples correspond to the case $n=3$.

\begin{proposition} \label{caso casi SIQ}
  Let $(C,\frak{p},\frak{d})$ be a $q$-cycle coalgebra. Assume that $\frak{p}_{11}^1 = 0$, $\frak{p}=\frak{d}$, $\frak{p}_{i0}^1=\delta_{1i}$ and that
  for all $j>0$, $\frak{p}_{1j}^1 = 0$. Then
  $$
    \frak{p}_{i1}^1=0,\quad\text{for all $i$.}
  $$
\end{proposition}

\begin{proof}
  Since $\frak{p}=\frak{d}$, the braid equations reduce to~\eqref{braid equation involutivo}. Since $\frak{p}_{i0}^1=\delta_{1i}$,
  when we consider the case $k=1$ and $i,j>0$, the equality~\eqref{braid equation involutivo} reads
  \begin{equation*}
    \sum_{a+b=j}\sum_{h=1}^{i} \frak{p}_{ia}^h \frak{p}_{1b}^1 \frak{p}_{h1}^1=
    \sum_{h=1}^{i} \frak{p}_{i1}^j \frak{p}_{j0}^j \frak{p}_{hj}^1+
    \sum_{l=1}^{j} \frak{p}_{i0}^i \frak{p}_{j1}^l \frak{p}_{il}^1.
  \end{equation*}
  By Proposition~\ref{formula para d1j1} we know that $\frak{p}_{j1}^j=j \frak{p}_{11}^1=0$, and by assumption $\frak{p}_{1b}^1=\delta_{b0}$,
  consequently this equality reduces to
  \begin{equation*}
    \sum_{h=2}^{i} \frak{p}_{ij}^h \frak{p}_{h1}^1=
    \sum_{h=2}^{i-1} \frak{p}_{i1}^h \frak{p}_{hj}^1+
    \sum_{l=1}^{j-1} \frak{p}_{j1}^l \frak{p}_{il}^1.
  \end{equation*}
  By~\eqref{producto}, we have
  $$
    \frak{p}_{ij}^i=\sum_{j_1+\dots+j_i=j}\prod_{s=1}^{i} \frak{p}_{1,j_s}^1=0,
  $$
  since $j>0$. Hence we arrive at
  \begin{equation}\label{estrella 1}
    \sum_{h=2}^{i-1} \frak{p}_{ij}^h \frak{p}_{h1}^1=
    \sum_{h=2}^{i-1} \frak{p}_{i1}^h \frak{p}_{hj}^1+
    \sum_{l=1}^{j-1} \frak{p}_{j1}^l \frak{p}_{il}^1.
  \end{equation}
  Now we prove that $\frak{p}_{i1}^1=0$ by induction on $i$. For $i=1$ this is true by assumption.
  Let $t\ge 1$ and assume that $\frak{p}_{i1}^1=0$ for all $i\le t$. Then $\frak{p}_{i1}^h=0$ for $i-t<h$. In fact, by~\eqref{producto},
  $$
    \frak{p}_{i1}^h=\sum_{\substack{ i_1+\dots+i_h=i\\ j_1+\dots+j_h=1}}\prod_{s=1}^{h} \frak{p}_{i_s,j_s}^1=0,
  $$
  since for some $s$ we have $j_s=1$, and so $\frak{p}_{i_s,j_s}^1=\frak{p}_{i_s,1}^1=0$, because $1\le i_s \le i-h+1 \le t$
  (Note that $i-t\le h-1$). So the equality~\eqref{estrella 1} reads
  $$
    \sum_{h=t+1}^{i-1} \frak{p}_{ij}^h \frak{p}_{h1}^1=
    \sum_{h=2}^{i-t} \frak{p}_{i1}^h \frak{p}_{hj}^1+
    \sum_{l=1}^{j-t} \frak{p}_{j1}^l \frak{p}_{il}^1.
  $$
  For $i=j=t+1$, only the second term of the right hand side survives, and we have
  $$
     0= \sum_{l=1}^{j-t} \frak{p}_{j1}^l \frak{p}_{il}^1=\frak{p}_{t+1,1}^1 \frak{p}_{t+1,1}^1,
  $$
  hence $\frak{p}_{t+1,1}^1=0$, which completes the inductive step and concludes the proof.
\end{proof}

\begin{example} \label{ej casi SIQ involutivo}
  For $n=3$ there exists a family of involutive $q$-cycle coalgebras parameterized by $\frak{p}_{22}^1,\frak{p}_{20}^1\in K$ with
  $\frak{p}_{10}^1 =1$ and $\frak{p}_{ij}^1=0$ if $(i,j)\notin\{ (1,0),(2,0),(2,2)\}$. In fact, one defines the coefficients
  $\frak{p}_{ij}^2$ using~\eqref{producto} and obtains that $\frak{p}_{20}^2 =1$ and $\frak{p}_{ij}^2=0$ if $(i,j)\ne (2,0)$.
   Then one verifies the equality~\eqref{braid equation involutivo} directly, or using a computer algebra system.
   If $\frak{p}_{20}^1=0$, this corresponds to the case considered in Proposition~\ref{caso casi SIQ}.
\end{example}

\begin{example}\label{ejemplo involutivo}
  For $n=3$ there exists a family of involutive $q$-cycle coalgebras parameterized by $\frak{p}_{12}^1,\frak{p}_{20}^1\in K$ with
  $$
  \frak{p}_{10}^1 =-1,\quad \frak{p}_{11}^1=0, \quad \frak{p}_{21}^1=2\frak{p}_{12}^1\quad\text{and}\quad
  \frak{p}_{22}^1=-\frac{5\frak{p}_{12}^1\frak{p}_{20}^1}{2}.
  $$
  In fact, one defines the coefficients
  $\frak{p}_{ij}^2$ using~\eqref{producto} and obtains that $\frak{p}_{20}^2 =1$, $\frak{p}_{22}^2=-2 \frak{p}_{12}^1$
  and $\frak{p}_{ij}^2=0$ if $(i,j)\notin\{(2,0),(2,2)\}$.
   Then one verifies the equality~\eqref{braid equation involutivo} directly, or using a computer algebra system.
\end{example}

\begin{example}\label{ejemplo no involutivo}
   For $n=3$ we find the following family of $q$-cycle coalgebras parameterized by
   $\frak{p}_{12}^1 \in K$,
   with $\frak{p}_{10}^1 =1$, $\frak{d}_{10}^1 =-1$,
   $$
     \frak{d}_{12}^1=-\frak{p}_{12}^1, \quad \text{and}\quad  \frak{p}_{ij}^1=\frak{d}_{ij}^1=0 \quad
     \text{if $(i,j)\notin\{(1,0),(1,2)\}$.}
   $$
   In fact, one defines the coefficients
   $\frak{p}_{ij}^2$ using~\eqref{producto} and verifies the equalities~\eqref{primera en m igual a 1}, \eqref{segunda en m igual a 1}
   and~\eqref{tercera en m igual a 1} directly, or using a computer algebra system.
   If we take $\frak{p}_{12}^1\ne 0$, then
   $$
     \frak{p}_{j0}^1=\delta_{1j}\quad\text{and}\quad \exists i=2>1, \frak{p}_{1i}^1\ne 0,
   $$
   but $\frak{p}\ne \frak{d}$ since $\frak{d}_{10}^1\ne \frak{p}_{10}^1$.
\end{example}

\begin{remark}
   A rather surprising fact is that in order to have a $q$-cycle coalgebra that is equivalent to a $\scc$ of degree $v_0=1$, it suffices to require
   $\frak{p}_{11}^1\ne 0$ (See Corollary~\ref{cor i0 igual a 1}). If we want to have a $q$-cycle coalgebra that is equivalent to a $\scc$ of degree $v_0\ge 1$, then by Corollary~\ref{cor i0 mayor que 1} it suffices to require
   \begin{enumerate}
     \item[a)] $\frak{p}=\frak{d}$,
     \item[b)] $\frak{p}_{10}^1=1$,
     \item[c)] $\frak{p}_{1j}^1\ne 0$, for some $0<j<n$.
   \end{enumerate}
   Example~\ref{ej casi SIQ involutivo} gives a $q$-cycle coalgebra satisfying a) and~b), but not~c); Example~\ref{ejemplo involutivo}
   gives a $q$-cycle coalgebra satisfying a) and~c), but not~b); and Example~\ref{ejemplo no involutivo}
   gives a $q$-cycle coalgebra satisfying b) and~c), but not~a). Thus the three conditions are necessary.
\end{remark}

\end{document}